\documentclass[11pt]{amsart}
\usepackage{amsthm, amssymb, geometry}
\usepackage[T1]{fontenc}

\usepackage{graphicx}
\begin{document}
\title{Growth of the Weil-Petersson Diameter of Moduli Space}

\author[W.~Cavendish]{William Cavendish}
\address[William Cavendish]{Department of Mathematics, Princeton University, USA}
\email{wcavendi@math.princeton.edu}

\author[H.~Parlier]{Hugo Parlier${}^\dagger$}

\address[Hugo Parlier]{Department of Mathematics, University of Fribourg\\
  Switzerland}
\email{hugo.parlier@gmail.com}
\thanks{${}^\dagger$Research partially supported by Swiss National Science Foundation grant number PP00P2\textunderscore 128557}

\maketitle

\theoremstyle{theorem}
\newtheorem{theorem}{Theorem}[section]
\newtheorem{thm}{Theorem}[section]
\newtheorem{prop}{Proposition}[section]
\newtheorem{lemma}{Lemma}[section]
\newtheorem{cor}{Corollary}[section]
\newcommand{\dc}{\textrm{d}_\mathcal{C}}
\newcommand{\fpg}{\mathcal{CP}_{g,0}}
\newcommand{\fpgn}{\mathcal{CP}_{g,n}}
\newcommand{\diam}{\mathrm diam}
\newcommand{\sys}{\mathrm{sys}\,}
\newcommand{\Mod}{\mathrm{Mod}\,}
\newcommand{\PP}{P}
\newcommand{\BB}{{\mathcal B}}
\newcommand{\R}{{\mathbb R}}
\newcommand{\MM}{{\mathcal M}}
\newcommand{\area}{{\rm Area}}

\begin{abstract}
\noindent In this paper we study the Weil-Petersson geometry of $\overline{\mathcal{M}_{g,n}}$, the compactified moduli space of Riemann surfaces with genus g and n marked points. The main goal of this paper is to understand the growth of the diameter of $\overline{\mathcal{M}_{g,n}}$ as a function of $g$ and $n$. We show that this diameter grows as $\sqrt{n}$ in $n$, and is bounded above by $C \sqrt{g}\log g$ in $g$ for some constant $C$.  We also give a lower bound on the growth in $g$ of the diameter of $\overline{\mathcal{M}_{g,n}}$ in terms of an auxiliary function that measures the extent to which the thick part of moduli space admits radial coordinates.
\end{abstract}
\section{Introduction}

The moduli space of curves $\mathcal{M}_{g}$ is the space of conformal structures on a topological surface $\Sigma$ of genus $g$. Via the uniformization theorem for surfaces, this space can be identified with the space of hyperbolic structures on $\Sigma$ up to isometry. $\mathcal{M}_{g}$ can also be realized as the quotient $\textrm{Teich}(\Sigma) / \Mod(\Sigma)$, where $\textrm{Teich}(\Sigma) $ is the Teichm\"uller space of marked hyperbolic structures on $\Sigma$ and $\Mod(\Sigma)$ is the mapping class group of $\Sigma$. In addition to being a parameter space for metrics on surfaces, $\textrm{Teich}(\Sigma) $ also admits several interesting $\Mod(\Sigma)$ invariant metrics of its own. These metrics descend to the quotient $\mathcal{M}_g$. In this paper we study the Weil-Petersson metric, a negatively curved $\Mod(\Sigma)$ invariant K\"ahler metric on $\textrm{Teich}(\Sigma)$. 

The Weil-Petersson metric on $\textrm{Teich}(\Sigma)$ is not complete, however the quotient of the metric completion by $\textrm{Mod}(\Sigma)$, as a topological space, is a well-known object called the Deligne-Mumford compactification of moduli space by stable nodal curves. In terms of hyperbolic structures, this compactification is given by adjoining ``strata" to moduli space whose points correspond to degenerate hyperbolic structures on $\Sigma$.  These degenerate hyperbolic structures are given by limits of sequences of hyperbolic structures on $\Sigma$ in which the hyperbolic length of some collection of disjoint simple closed geodesics goes to zero. These strata are lower dimensional moduli spaces parameterizing families of cusped hyperbolic surfaces, and many geometric and topological properties of $\mathcal{M}_{g}$ can be understood inductively using properties of the strata. Since the completion of the Weil-Petersson metric on $\mathcal{M}_g$ is a compact space, the Weil-Petersson diameter of $\mathcal{M}_g$ is finite. This paper studies the growth of $\textrm{diam}(\mathcal{M}_{g})$, and more generally the diameter of $\mathcal{M}_{g,n}$, the space of hyperbolic structures on a surface of genus $g$ with $n$ punctures. The main results are the following:

\begin{thm}\label{thm:ngrowth} There exists a genus independent constant $D$ such that
\[
 \lim_{n\to\infty} \frac{\emph{diam}(\mathcal{M}_{g,n})}{\sqrt{n}}=D.
\]
\end{thm}

\begin{thm}\label{thm:ggrowth} There exist a constant $C>0$ such that for any $n\ge 0$ the Weil-Petersson diameter $\emph{diam}(\mathcal{M}_{g,n})$ satisfies 
\[
\frac{1}{C}\le \liminf_{g\to\infty}\frac{\emph{diam}(\mathcal{M}_{g,n})}{\sqrt{g}},~
\limsup_{g\to\infty}\frac{\emph{diam}(\mathcal{M}_{g,n})}{\sqrt{g}\log(g)}
\le C.
\]
\end{thm}

The strategy for establishing upper bounds on $\textrm{diam}(\mathcal{M}_{g,n})$ is to study the collection of maximally noded surfaces (surfaces in which a maximal disjoint collection of non-isotopic simple closed curves, a pants decomposition, has been pinched) following the approach of Brock in \cite{pd}. The first step is to bound the distance from an arbitrary point in $\overline{\mathcal{M}_{g,n}}$ to a maximal node. This is done using a recursive argument depending on Wolpert's estimate on the lengths of pinching rays, and upper bounds on the diastole of a hyperbolic surface. We then bound the distance between maximal nodes in terms of distances in a combinatorial object we call the \emph{cubical pants graph} $\fpgn$, which is given by adding diagonals of multi-dimensional cubes to the pants graph $\mathcal{P}_{g,n}$.  The lower bounds on diameter implied by the above theorems are a simple consequence of the geodesic convexity and product structure of the strata.

Though the methods of this paper are unable to resolve the $\log(g)$ disparity between the upper and lower bounds in Theorem \ref{thm:ggrowth}, by combining volume estimates for the thick part of $\mathcal{M}_{g,n}$ due to Schumacher-Trapani and Mirzakhani with Ricci curvature estimates on the thick part of $\mathcal{M}_{g,0}$ due to Teo, we are able to produce a lower bound on the diameter of $\mathcal{M}_{g,n}$ in terms of an geometrically defined auxiliary function $v_\varepsilon(g)$ depending on a parameter $\varepsilon$.  Roughly speaking, this function measures how far into the thin part of $\mathcal{M}_g$ minimal length geodesics dip when traveling between points in the thick part of $\mathcal{M}_g$. 

\begin{thm}\label{thm:conditionalggrowth} There exist a constant $C>0$ such that for any $n\ge 0$ the Weil-Petersson diameter $\emph{diam}(\mathcal{M}_{g,n})$ satisfies 
\[
\frac{1}{C}\le \liminf_{g\to\infty}\frac{\emph{diam}(\mathcal{M}_{g,n})}{v_\varepsilon(g)\sqrt{g}\log(g)}
\]
\end{thm}

We will define the function $v_\varepsilon(g)$ in section \ref{sec:LB}, but we remark that if an $\varepsilon$ exists such that the $\varepsilon$-thick part of $\mathcal{M}_{g,0}$ has a star-shaped fundamental domain, then $v_\varepsilon(g)=\varepsilon$.

The results of this paper can be viewed as progress towards understanding the ``intermediate-scale" geometry of Teichm\"uller space with the Weil-Petersson metric.  Brock's work in \cite{pd}, which establishes the existence of a quasi-isometry between Teichm\"uller space  with the Weil-Petersson metric and the pants graph, gives a very good understanding of the Weil-Petersson geometry of a given Teichm\"uller space in the large.  The main thrust of this paper is to understand the extent to which Brock's methods can be adapted to understand the Weil-Petersson metric on a smaller scale.  Theorems  \ref{thm:ggrowth} and \ref{thm:ngrowth}, which address a question of Wolpert in \cite{Wolp3}, show that $\fpgn$ can be used to model the Weil-Petersson metric well on a scale comparable to the diameter of moduli space.  

This paper is organized as follows. Section \ref{sec:BM} provides the relevant background on Teichm\"uller space and the Weil-Petersson metric. Section \ref{sec:MN} provides upper bounds on the distance from an arbitrary point in the interior of Teichm\"uller space to a maximally noded surface, and section \ref{sec:MCP} introduces the cubical pants graph $\fpgn$ and uses it to bound the distance between maximally noded surfaces. Section \ref{sec:LB} establishes the lower bounds bounds in the above theorems, and section \ref{sec:GS} establishes the existence and genus independence of the limit in Theorem \ref{thm:ngrowth}.

{\bf Acknowledgements.} We would like to thank a number of people for interesting discussions concerning different aspects of this paper including Florent Balacheff, Jeff Brock, Zeno Huang, Maryam Mirzakhani, Kasra Rafi, Juan Souto, and Scott Wolpert. In particular we are grateful to Jeff Brock for providing many important insights and for encouraging us to pursue the problem, and to Maryam Mirzakhani for providing us with the statement of Theorem \ref{thm:mirz} and for pointing out an error in a previous draft of our paper.

\section{Background Material}\label{sec:BM}

\subsection{The Weil-Petersson metric on Teichm\"uller space}

The Teichm\"uller space of an orientable surface $\Sigma$ of negative Euler characteristic with genus $g$ and $n$ punctures is the set of marked hyperbolic metrics on $\Sigma$. More formally,
\[
\textrm{Teich}(\Sigma) = \{\varphi : \Sigma \to S~|~ S~\textrm{is a finite area hyperbolic surface},~\varphi~\textrm{is a homeomorphism} \} /\sim,
\]
where $\varphi_1\sim\varphi_2$ if $\varphi_1\circ\varphi_2^{-1}$ is isotopic to the identity. The map $\varphi$ from $\Sigma$ to $S$ is called a marking, and given $\psi\in \textrm{Homeo}^+(S)$, the group of orientation preserving homeomorphisms of $\Sigma$, we get a map $\textrm{Teich}(\Sigma)\to\textrm{Teich}(\Sigma)$ given by $(\varphi: \Sigma\to S)\mapsto(\varphi\circ\psi: \Sigma \to S)$. By the equivalence relation $\sim$, this map is trivial if $\psi\in \textrm{Homeo}^0(S)$, the normal subgroup of homeomorphisms isotopic to the identity, so precomposition gives an action of the mapping class group $\Mod(\Sigma)=\textrm{Homeo}^+(\Sigma)/\textrm{Homeo}^0(\Sigma)$ on $\textrm{Teich}(\Sigma)$.

$\textrm{Teich}(\Sigma)$ is homeomorphic to $\mathbb{R}^{6g-6+2n}$, and can be given global coordinates as follows. Given a collection $P$ of $3g-3+n$ disjoint non-isotopic simple closed curves on $\Sigma$, $\Sigma\setminus P$ will have $2g-2+n$ components each of which is homeomorphic to a 3-holed sphere. Such a collection is called a pants decomposition. It is an elementary theorem in hyperbolic geometry that given any triple of numbers $(a,b,c)$ there is a unique hyperbolic structure on the 3-holed sphere having geodesic boundary components of lengths $(a,b,c)$. Thus given a topological surface $\Sigma$ together with a pants decomposition $P$, to specify a hyperbolic structure on $\Sigma$ we need to specify $3g-3+n$ positive real numbers for the lengths of curves in $P$, together with $3g-3+n$ real numbers to indicate how the pairs of pants are glued together (for details on this construction see for instance \cite{Bus}). This gives the Fenchel-Nielsen coordinate system $\textrm{Teich}(\Sigma)\to (\mathbb{R}^+)^{3g-3+n}\times\mathbb{R}^{3g-3+n} $, where the first $3g-3$ coordinates are called length coordinates, and the last $3g-3$ coordinates are called twist coordinates. 

The cotangent space at $X\in\textrm{Teich}(\Sigma)$ can be identified with $\mathcal{Q}(X)$, the space of holomorphic quadratic differentials on $X$. In a local coordinate $z$ on $X$, a quadratic differential $\phi$ has the form $h(z) dz^2$, where $h(z)$ is a holomorphic function. Dual to quadratic differentials are the Beltrami differentials $\mu\in\mathcal{B}(X)$, which can be written in local coordinates as $f(z)d\bar{z}/dz$. Given $\phi\in\mathcal{Q}(X)$ and $\mu\in\mathcal{B}(X)$, $\phi\mu$ has a coordinate expression of the form $f(z)|dz|^2$, which is an area element that can be integrated over $X$. We therefore get a natural pairing between holomorphic quadratic differentials and Beltrami differentials given by $(\phi,\mu)=\int_X\phi\mu$. Let $\mathcal{Q}(X)^\perp\subset \mathcal{B}(X)$ denote the subspace of $\mathcal{B}(X)$ that is perpendicular under this pairing to $\mathcal{Q}(X)$. The tangent space at $X\in\textrm{Teich}(\Sigma)$ is given by $\mathcal{B}(X)/\mathcal{Q}(X)^\perp$. 

By the unformization theorem, $X$ has a unique hyperbolic metric with line element $\rho$ that can be written in local coordinates as $g(z) |dz|$. Given $\phi, \psi\in\mathcal{Q}(X)$, $\phi\bar{\psi}/\rho^2$ is an area element so we can define a Hermitian inner product on $\mathcal{Q}(X)$ by
\[
\langle \phi,\psi \rangle =\int_X \frac{\phi\bar{\psi}}{\rho^2}
\]
Dualizing via the pairing $(\cdot~,\cdot)$, we get a Hermitian pairing on the tangent space $\mathcal{B}(X)/\mathcal{Q}(X)^\perp$ whose real part gives a positive definite inner product. To get an inner product in the tangent space, we define a norm on $\mathcal{B}(X)/\mathcal{Q}(X)^\perp$ by $||\mu||_{WP}=\sup_{\{\phi~|~\langle\phi,\phi \rangle=1\}}(\mu,\phi)$, and an inner product via polarization: $\langle\mu,\nu\rangle_{WP}=1/4(||\mu+\nu||_{WP}^2-||\mu-\nu||_{WP}^2)$. The resulting Riemannian metric $g_{WP}$ is called the Weil-Petersson (WP) metric, and since the definition of the metric depends only on the holomorphic structure at $X$ and not the marking, this metric is $\Mod(\Sigma)$ invariant.

The resulting metric has many nice properties, some of which we will outline here. For a more thorough survey of this material we direct the reader to \cite{Wolp2,Wolp3}. Teichm\"uller space equipped with the Weil-Petersson metric is a unique geodesic metric space of negative curvature, and $g_{WP}$ is a K\"ahler metric whose K\"ahler form $\omega$ is given in the Fenchel-Nielsen coordinates defined by a pants decomposition $P$ by $\omega=\sum_{\gamma\in P} d\ell_{\gamma}\wedge d\tau_{\gamma}$, where $\ell_\gamma$ is the length coordinate associated to the curve $\gamma\in P$, and $\tau_\gamma$ is the twist coordinate. This metric space is non-complete, since simple closed curves can be pinched down to a pair of cusps in finite time. The following theorem of Wolpert quantifies this statement (see section 4 in \cite{Wolp1}).

\begin{thm}[Length of Pinching Rays]\label{pinchlength} Let $X_0\in \emph{Teich}(\Sigma)$ and $\mu$ a multicurve of length $L$ on $X$, and $\ell_\mu:\emph{Teich}(\Sigma)\to \mathbb{R}$ the function assigning to $X$ the length of $\mu$ on $X$. Then there exists a path $\gamma:[0,1)\to\emph{Teich}(\Sigma)$ with $\emph{length}(\gamma)\le\sqrt{2\pi L}$ such that $\gamma(0)=X_0$ and $\lim_{t\to1} \ell_{\mu}(\gamma(t))=0$. 
\end{thm}

A theorem of Masur \cite{Mas} shows that the finiteness of pinching rays is the only source of non-completeness for the Weil-Petersson metric on Teichm\"uller space, and that by adjoining strata corresponding to Teichm\"uller spaces of cusped surfaces we obtain a complete metric space with stratified boundary called augmented Teichm\"uller space (see \cite{Ab}). The structure of this completion is easy to understand in terms of Fenchel-Nielsen coordinates. By letting the length coordinate for a curve $\gamma\in P$ take the value zero in a chart corresponding to a pair of pants $P$ and ignoring the twist parameter about $\gamma$, we obtain $\mathcal{S}_\gamma=\{X~:\ell_\gamma(X)=0\}$, the stratum corresponding to $\gamma$. More generally, given a multicurve $\mu \subset P$, the stratum $\mathcal{S}_\mu$ is given by letting all the curves in $\mu$ have length $0$ and ignoring the twisting about $\mu$. This provides charts
\[
(\mathbb{R}^{\ge0})^{3g-3+n}\times \mathbb{R}^{3g-3+n}/\sim~\to\overline{\textrm{Teich}(\Sigma)}
\] 
where $(\ell_1,\cdots,\ell_{3g-3+n},\tau_1,\cdots,\tau_{3g-3+n})\sim(\ell_1',\cdots,\ell_{3g-3+n}',\tau_1',\cdots,\tau_{3g-3+n}')$ if $\ell_i=\ell_i'$ for all $i$ and $t_i=t_i'$ for all $i$ such that $\ell_i\ne0$. 

Note that if $\mu$ is a multicurve such that $\Sigma \setminus\mu$ has $k$ components, then $\mathcal{S}_\mu$ is a product of $k$ Teichm\"uller spaces of lower dimension. Masur's work in \cite{Mas} shows that each stratum inherits a Riemannian metric which respects this product structure (see also \cite{Wolp2}).

\begin{thm}[Product Structure of Strata]\label{product structure} Let $\Sigma_{g',n'}$ denote a topological surface of genus $g'$ with $n'$ punctures. Let $\mu$ be a multicurve on $\Sigma_{g,n}$ such that 
$$\Sigma \setminus \mu \cong \Sigma_{g_1,n_1}\sqcup\Sigma_{g_2,n_2}\sqcup\cdots\sqcup\Sigma_{g_k,n_k},$$
 and let $\mathcal{S}_\mu$ be the stratum defined by $\mu$. The metric induced on $\mathcal{S}_{\mu}$ as a subset of the Weil-Petersson completion of $\emph{Teich}(\Sigma_{g,n})$ is isometric to the Riemannian product metric 
$$\left(\emph{Teich}(\Sigma_{g_1,n_1}),g_{WP}\right)\times \cdots\times \left(\emph{Teich}(\Sigma_{g_k,n_k}),g_{WP}\right).$$
\end{thm}
An important property of these strata is the following theorem of Yamada \cite{Yam} building on work of Wolpert \cite{Wolp1}.

\begin{thm}[Convexity of Strata]\label{convexity}The closure of each stratum in $\overline{\emph{Teich}(\Sigma)}$ is geodesically convex. 
\end{thm}

\subsection{The Weil-Petersson metric on moduli space}

As mentioned before, the quotient $\overline{\textrm{Teich}(\Sigma)}/\Mod(\Sigma)$ is precisely $\overline{\mathcal{M}_{g,n}}$, the Deligne-Mumford compactification of moduli space \cite{Mas}. From the above discussion, one can see that the quotients of the strata of $\overline{\textrm{Teich}(\Sigma)}$ by the mapping class group yield strata of the $\overline{\mathcal{M}_{g,n}}$. These strata are products of covering spaces of lower dimensional moduli spaces. Much of the work of this paper goes into understanding combinatorial questions about how these strata fit together.

\begin{figure}[htbp]
 \centering
 \includegraphics[width=4 in]{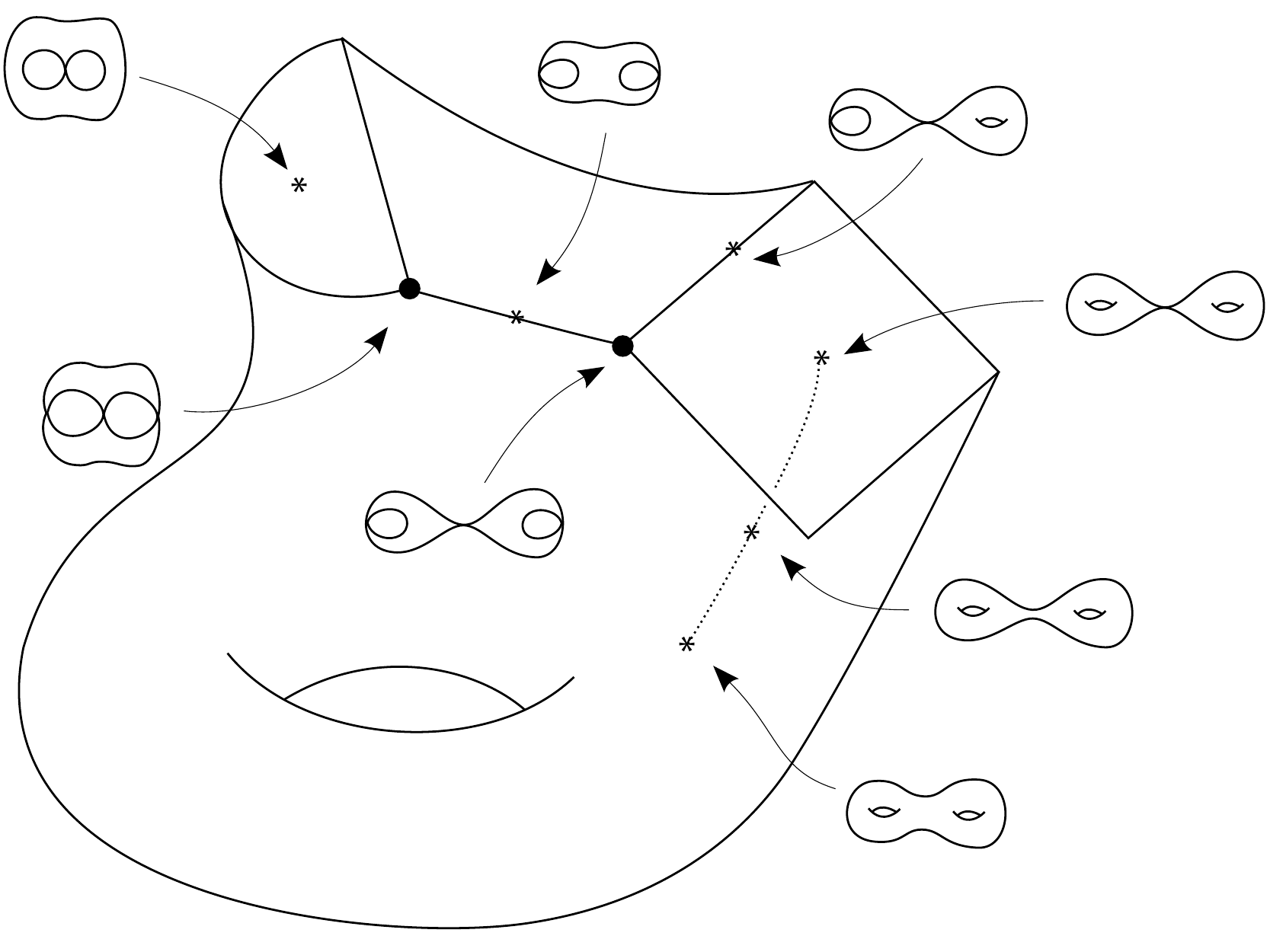}
 \caption{A caricature of the Deligne-Mumford compactification}
 \label{fig:DM}
\end{figure}

A result of Wolpert \cite{WolpCohom} shows that the cohomology class of the Weil-Petersson symplectic form is a constant multiple of a cohomology class that shows up naturally in algebraic geometry, called the first tautological Mumford class. Thus the WP Riemannian volume form has an algebro-geometric interpretation and one can use the methods of algebraic geometry to understand the Weil-Petersson volume of moduli space \cite{Grush,Mirz2,Schu}. The asymptotics of the volume of this space as genus grows are given by the following theorem of Schumacher and Trapani \cite{Schu}.

\begin{thm}[Volume growth]\label{thm:schu} The volume of moduli space satisfies
\[
\lim_{g\to\infty} \frac{\log \emph{Vol}(\mathcal{M}_{g,n})}{g \log g}=2
\]
for any fixed $n$.
\end{thm}

Our approach to proving lower bounds for the diameter involves using these volume asymptotics together with estimates due to Teo \cite{Teo} on the Ricci curvature of the Weil-Petersson metric. These bounds at a point $X$ depend on the hyperbolic geometry of the surface $X$, and in particular on the injectivity radius of $X$. Let $\mathcal{M}_{g,n}^\varepsilon$ denote the $\varepsilon$-thin part of moduli space, i.e. the set of all surfaces with a homotopically non-trivial curve of length less than $\varepsilon$, and let $\mathcal{M}_{g,n}^{T(\varepsilon)}$ denote the complement of this set, the $\varepsilon$ thick part of moduli space.

\begin{thm}[Ricci curvature bound]\label{thm:teo} The Ricci curvature of $\mathcal{M}_{g,0}^{T(\varepsilon)}$ is bounded below by $-2 C(\varepsilon)^2$, where
\[
C(\varepsilon)=\left\{\frac{4\pi}{3}\left[1-\left(\frac{4e^{\varepsilon}}{(e^\varepsilon+1)^2}\right)^3\right]\right\}^{-\frac{1}{2}}
\]
\end{thm}

As Teo remarks, $C(\varepsilon)\sim \frac{1}{\sqrt{\pi} \varepsilon}+O(1)$, so the Ricci curvature of the WP metric blows up as one approaches the strata. It will be important for our purposes to know that the volume of moduli space is not concentrated near the strata where the Ricci curvature is badly behaved. This is guaranteed by a recent result of Mirzakhani \cite{Mirz}.

\begin{thm}[Volume of the thick part of moduli space]\label{thm:mirz} Given $c>0$ and $n>0$, there exists $\varepsilon>0$ such that for any $g\ge 2$
\[
\emph{Vol}(\mathcal{M}_{g,n}^{\varepsilon})<c \emph{Vol}(\mathcal{M}_{g,n})
\]
\end{thm}

The estimates used in Mirzhakani's proof of this theorem had been predicted in \cite{Zog}. Further indications that this result is true can be found in the study of random surfaces. In \cite{brma04}, a random surface is defined by considering conformal structures that arise from random gluings of ideal triangles. Brooks and Makover \cite{brma04} show that surfaces have a systole of length bounded below by some universal constant with probability $1$. However, although random surfaces are dense in moduli space, it is not known how they are distributed with respect to the Weil-Petersson metric.

\subsection{Short curves and multicurves on surfaces}

Short simple closed curves, or collections of short curves, have been studied by both hyperbolic and Riemannian geometers for many years. The systole, i.e. the length of the shortest non-homotopically trivial curve, has been of particular interest. From the discussion above, it is clear that the Weil-Petersson metric is sensitive to this quantity.

Understanding the lengths of pants decompositions on hyperbolic surfaces is also important in what follows.  These lengths are controlled by quantities called Bers' constants \cite{be74,be85}, which are upper bounds on the lengths of all curves in a pants decomposition. Formally, for a given surface $S\in\mathcal{M}_{g,n}$ and a pants decomposition $\PP$ of $S$, we define the length of $\PP$ as 
$$
\ell(\PP)=\max_{\gamma \in \PP} \ell(\gamma).
$$
We denote $\BB(S)$ the length of a shortest pants decomposition of $S$ and the quantity $\BB_{g,n}$ is defined as
$$
\BB_{g,n}=\sup_{S\in \MM_{g,n}} \ell(\BB(S)).
$$
This quantity is well defined by a theorem of Bers \cite{be74}, who showed that for any given hyperbolic surface the length of its shortest pants decomposition can be bounded by a function that only depends on its topology.  Much work has been put into quantifications of Bers' constants, in particular by Buser \cite{Bus}.

We require the following bounds in our analysis, which determine the rough growth of the constants as a function of the number of cusps. The first is a bound on Bers' constants for punctured spheres \cite{Bal2}.

\begin{thm} [Bers' constants for punctured spheres] \label{thm:spherebers}Bers' constants for punctured spheres satisfy 
\[
\mathcal{B}_{0,n}\le 30\sqrt{2\pi(n-2)}
\]
\end{thm}

This can be generalized to surfaces of genus $g$ with $n$ punctures as follows \cite{Bal3}.

\begin{thm}[Bers' constants for punctured surfaces]\label{thm:BPS} There exists a function $A(g)$ such that 
$$
\mathcal{B}_{g,n} \leq A(g)\sqrt{n}.
$$
\end{thm} 

For our purposes it will also be necessary to find short multicurves which split the surface into two parts of more or less equal area. To do this, we require the following theorem due to Balacheff and Sabourau \cite{Bal}, which gives a bound on a geometric quantity called the {\it diastole} of a surface. Given a Riemannian surface $M$, the diastole of $M$ is defined to be the infimum, taken over all functions $\phi:M\to \R$, of the supremum of $\{\textrm{length}(\phi^{-1}(t))\}$. The main theorem of \cite{Bal} can be phrased as follows.

\begin{thm}[Diastolic inequality]\label{thm:diast} There exists a constant $C$ such that if $M$ is a closed Riemannian surface of genus $g$, then
$$\rm{diastole}(M) \leq C \sqrt{g+1}\sqrt{\area(M)}.
$$
\end{thm}

Note that the theorem can also be used for finite area complete hyperbolic surfaces, by the following trick. One can cut off the cusps of the hyperbolic surface and replace the tips with small hemispheres to obtain a closed Riemannian surface $M$ with almost the same area. The theorem applies to $M$, so a simple argument shows that it also applies to the original surface. We also point out that ${\rm diastole} \over \area/2$ provides an upper bound on the Cheeger constant of a surface \cite{ch70}, which is in turn bounded below by ${1\over 10}(\sqrt{1+10 \lambda_1}-1)$ by Buser's inequality \cite{BusIneq}, where $\lambda_1$ is the first eigenvalue of the Laplacian on the surface. Since there exist families of hyperbolic surfaces with $\lambda_1$ uniformly bounded below by a positive constant, one cannot hope to improve the diastolic inequality by either a stronger genus or area factor. Examples of such families of surfaces can be found in \cite{br99}.

\subsection{The pants graph}

As mentioned above, understanding the combinatorics of the strata of moduli space will be important in what follows. The strata of complex dimension 1 will be particularly important. In Teichm\"uller space, the combinatorics of the 1-dimensional strata can be encoded using a well studied object $\mathcal{P}_{g,n}$, introduced in \cite{hath80}, called the pants graph. 

Vertices in this graph are topological types of pants decompositions, and two pants decompositions $P_1$ and $P_2$ span an edge if they are joined by a so-called elementary move.  These moves are given by removing a single curve from the pants decomposition and replacing it by a curve that intersects it minimally,  as shown in figure \ref{fig:moves}.  Given a pants decomposition, there is a corresponding maximally noded surface in Teichm\"uller space.  Maximally noded surfaces whose pants decompositions differ by an elementary move sit on the boundary of a single 1-dimensional stratum, and we can move between such maximal nodes by deforming the surface in a space that is isometric either to the Teichm\"uller space of the 4-holed sphere or the Teichm\"uller space of the once punctured torus. Thus strata of $\mathbb{C}$-dimension one correspond to the edges in the pants graph.

\begin{figure}[htbp]
 \centering
 \includegraphics[width=6 in]{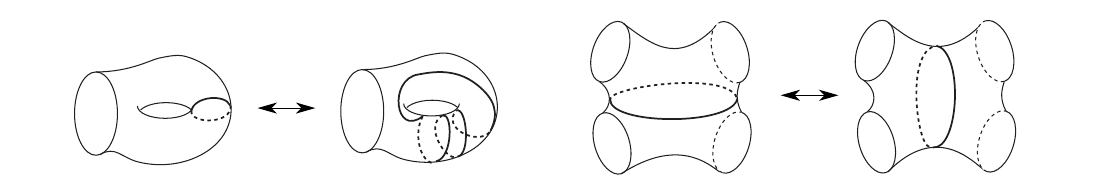}
 \caption{Pants moves}
 \label{fig:moves}
\end{figure}

By assigning length $1$ to each edge, this graph becomes a metric space. Given a surface in $\textrm{Teich}(\Sigma)$, one associates to it one of its short pants decompositions (those of length less than Bers' constant $\BB_{g,n}$). This provides a coarsely defined map from $\textrm{Teich}(\Sigma)$ to $\mathcal{P}_{g,n}$. Brock \cite{pd} showed that this map provides a quasi-isometry between $\mathcal{P}_{g,n}$ and $\textrm{Teich}(\Sigma)$ with the Weil-Petersson metric. The quasi-isometry Brock provides, however, is not uniform in $g$ and $n$ and must be modified to study the diameter growth of moduli space. We will show in section \ref{sec:MCP} that by adding edges of appropriate length to $\mathcal{P}_{g,n}$, one can form a 1-complex $\mathcal{CP}_{g,n}$ that can be used to get uniform control on the geometry of the $1$-skeleton of moduli space. 


\section{Bounding the Distance to a Maximal Node}\label{sec:MN}

The goal of this section is to prove the following lemmas:

\begin{lemma}\label{spherepinch} Let $S$ be a point in $\mathcal{M}_{0,n}$. The distance from $S$ to some maximal node is bounded above by $C \sqrt{-\chi(S)}$. 
\end{lemma}

\begin{lemma}\label{genuspinch} Let $S$ be a point in $\mathcal{M}_{g,n}$. The distance from $S$ to some maximal node is bounded above by $C \sqrt{-\chi(S)}\log(-\chi(S))$. 
\end{lemma}

The central piece of input in this section is Wolpert's pinching estimate (Theorem \ref{pinchlength} above), which shows that a short multicurve on a surface $X$ can be pinched quickly. Given this estimate, a naive approach to getting to a maximal node quickly would be to find a short pants decomposition on $X$ and apply Theorem \ref{pinchlength}. This approach, however, does not give the required upper bound. Instead of pinching all the curves in a pants decomposition at once, it turns out to be advantageous to pinch short separating curves that cut the surface roughly into equal parts, and then use the product structure of the boundary and an inductive argument.

\subsection{Pinching punctured spheres}

We first prove Lemma \ref{spherepinch}. Recall Theorem \ref{thm:spherebers} from the preliminaries, which gives that any $n$-punctured sphere has a pants decomposition where all the curves have length at most $30\sqrt{2\pi(n-2)}$.

This theorem, taken together with the following lemma, shows that each punctured sphere contains a short curve which cuts it into pieces with roughly the same size.

\begin{lemma}\label{sepcurve} In any pants decomposition of an $n$-times punctured sphere $S$, there exists a curve $\gamma$ such that each component of $S\setminus\gamma$ has at least $n/3$ punctures.
\end{lemma}

\begin{proof} Given a pair of pants $P$ in a pants decomposition of $S$ with boundary components $\gamma_1$, $\gamma_2$ and $\gamma_3$, each of the $\gamma_i$'s that is not a cusp itself separates $S$. Let $p_i$ be the number of punctures in the component of $S\setminus\gamma_i$ that does not contain $P$, and set $p_i=1$ if $\gamma_i$ is a cusp. Note that for each $P$, $p_1+p_2+p_3=n$. We assume $p_1$ is the largest of the $p_i$'s and we will call this the weight of $P$. Suppose that $p_1>p_2+p_3$. Then we may replace $P$ by $P'$, the pair of pants which shares the cuff $\gamma_1$ with $P$. The number of punctures cut off by the boundaries of $P'$ have lengths $p_1'=p_2+p_3$, $p_2'$, and $p_3'$ for some $p_2',p_3'>0$ s.t. $p_2'+p_3'=p_1$. Thus we obtain a pair of pants with strictly smaller weight. By iterating this procedure, we end up with a pair of pants $P''$ such that each of $p_1'',p_2''$ and $p_3''$ is at most the sum of the other two. If $p_1''$ is the largest $p_i''$ then $p_i''\ge n/3$, and since $p_1''\le p_2''+p_3''=n-p_1''$, $p_1''$ is at most $n/2$. Thus each components of $\gamma_i''$ cuts off at least $n/3$ punctures.
\end{proof}

We will require the following technical lemma.

\begin{lemma}\label{recursiveineq1} Let $F:\mathbb{N}\to\mathbb{R}$ be a monotonic function satisfying
\[
F(n)\le\sqrt{ F(\lambda n+1)^2+F((1-\lambda )n+1)^2}+n^{1/4}
\] 
for $\lambda=\lambda(n)\in[1/3,2/3]$. Then $F(n)\le C n^{1/2}$ for some constant $C$. 
\end{lemma}

\begin{proof}

We will show by induction that the stronger inequality $F(n)\le C(n^{1/2}-n^{1/3})$ holds for some $C$. Suppose this is the case for all $m< n$.

Let $G(n,x)=\left((nx+1)^{1/2}-(nx+1)^{1/3}\right)^2$. We have by induction that
\[
F(n)\le \sqrt{ F(\lambda n+1)^2+F((1-\lambda )n+1)^2}+n^{1/4}\le C\sqrt{ G(n,\lambda)+G(n,1-\lambda)}+n^{1/4}
\]
It is easy to check that the function $(\partial/\partial x)^2G(n,x)$ is strictly positive on the interval $[1/3,2/3]$, which implies that $G(n,x)+G(n,1-x)$ has a unique critical point at $x=1/2$ that is a local minimum. This shows that $G(n,\lambda)+G(n,1-\lambda)\le G(n,1/3)+G(n,2/3)$, so
\[
F(n)\le C\sqrt{ G(n,1/3)+G(n,2/3)}+n^{1/4}.
\]
Expanding the term under the radical, we obtain a term of the form $\sqrt{n-\psi(n)}$, where $ \psi(n)$ has leading order term $a\cdot n^{5/6}$, where $a=2\cdot3^{-5/6}(1+2^{5/6})>2$. Thus for sufficiently large $n$, we may conclude by Taylor expansion that
\[
n^{-1/2}F(n)\le C\sqrt{1-\frac{\psi(n)}{n}}+n^{-1/4} \le C \left(1-\frac{1}{2}\frac{\psi(n)}{n}\right)+n^{-1/4}\le C\left(1-n^{-1/6}\right)
\]
This shows that there exists an integer $N$ such that 
$$F(n)\le C\left(n^{1/2}-n^{1/3}\right)\Rightarrow F(n+1)\le C\left((n+1)^{1/2}-(n+1)^{1/3}\right)$$
for all $n\ge N$. By choosing $C$ large enough, we can ensure that all base cases are satisfied.
\end{proof}

\begin{proof}[Proof of Lemma \ref{spherepinch}]

Let $F(n)$ be the maximal distance from an $n$-times punctured sphere in Teichm\"uller space to a maximally noded surface. Let $X$ be an arbitrary surface. By Theorem \ref{thm:spherebers}, $X$ has a pants decomposition in which each curve has length at most $30\sqrt{2\pi(n-2)}$. By Lemma \ref{sepcurve}, one of the curves in this decomposition, $\gamma$, separates $X$ into components each of which has at least $n/3$ punctures. By Theorem \ref{pinchlength}, the distance from $X$ to some point $X'$ in $\mathcal{S}_\gamma$ is at most $\sqrt{2\pi~\textrm{length}(\gamma)}\le D n^{1/4}$. The metric on the stratum corresponding to $\gamma$ is a product, so the distance from $X'$ to a maximal node is at most $\sqrt{F(a)^2+F(b)^2}$ where $a$ and $b$ are the number of punctures on the components of $X'$. These numbers satisfy $a=\lambda n+1$, $b=(1-\lambda)n+1$ for some $\lambda\in [1/3,2/3]$, since $X'$ has two more cusps than $X$. 

Since this holds for any $X$, there exists $\lambda=\lambda(n)\in[1/3,2/3]$ such that $F$ satisfies a recursive upper bound of the form
\[
F(n)\le\sqrt{ F(\lambda n+1)^2+F((1-\lambda)n+1)^2}+D n^{1/4},
\] 
so $F(n)/D$ satisfies the hypotheses of Lemma \ref{recursiveineq1}.
\end{proof}

\subsection{Pinching surfaces with genus}

The proof in the case of non-planar surfaces is similar to the case for punctured spheres, though in this case short separating curves are more difficult to find. We begin by showing the following lemma, which is an application of Theorem \ref{thm:diast} from the preliminaries.

\begin{lemma}\label{hypdiast} Let $S$ be a finite area hyperbolic surface. Then $S$ has a separating geodesic multicurve $\mu$ of length at most $C \chi(S)$ such that $S\setminus\mu=A\amalg B$, and $\emph{Area}(A)\geq \emph{Area}(B)\geq \emph{Area}(S)/2-\pi$.
\end{lemma}

\begin{proof}[Proof of Lemma \ref{hypdiast}]
Assume that $S$ is closed. By Theorem \ref{thm:diast}, there exists a smooth function $\phi: S\to\mathbb{R}$ such that for all $t\in\mathbb{R}$, the total length of the level set $\phi^{-1}(t)$ is at most $C\sqrt{g+1}\sqrt{\textrm{Area}(S)}$, which is at most $D\cdot(-\chi(S))$ for some constant $D$ by the Gauss-Bonnet theorem. By perturbing $\phi$ if necessary, we may assume that the critical points of $\phi$ are isolated and $\phi$ takes a different value on each critical point. Thus for all but finitely many $t$, $\phi^{-1}(t)$ is a multicurve. Let $\{t_0,t_1,\cdots,t_k\}$ be the list of exceptional $t$.  By tightening each component of $\phi^{-1}(t)$ to a geodesic (or to a point in the case that the component is homotopically trivial) we obtain a geodesic multicurve (possibly with multiplicities) whose total length is less that the length of $\phi^{-1}(t)$. Note that distinct components of the level set may be homotopic in the surface, and therefore will collapse to the same geodesic. For $t$ a regular value, let $S_t$ be the surface given by taking the subsurface $\phi^{-1}((-\infty,t])$, throwing away disk and annuli components and tightening the remaining boundary components to geodesics. Note that boundary curves that are homotopic in $X$ will tighten to the same geodesic, so topologically $S_t$ is given by taking $\phi^{-1}((-\infty,t])$, throwing away annuli and disks, gluing a disk in to any boundary component of $\phi^{-1}((-\infty,t])$ that is homotopically trivial, and gluing any two boundary components of $\phi^{-1}((-\infty,t])$ that are homotopically trivial together. 

Given $t,t'\in(t_i,t_{i+1})$, $\phi^{-1}(t)$ and $\phi^{-1}(t')$ are homotopic, so the topology of $S_t$ can only change when $t$ crosses a critical value. Since the critical points are isolated and take different values of $t$, $\phi^{-1}(t_i)$ is a union of circles and either one point, or one figure-eight (i.e. wedge of two circles). It is easy to see that the topology of $S_t$ doesn't change by crossing a point singularity. A neighborhood of a figure-eigth level set is a pair of pants, so it is clear that crossing one of these level sets adds at most one pair of pants to $S_t$, and since a pair of pants has hyperbolic area $\pi$ the area of $S_t$ changes by at most $\pi$ each time a singularity is crossed. Since the area of $S_t$ goes from $0$ to $\textrm{Area}(S)$, at some time $t$, $\textrm{Area}(S)/2\le \textrm{Area}(S_t)\le \textrm{Area}(S)/2+\pi$. Since the boundary $\partial S_t$ of $S_t$ is given by tightening $\phi^{-1}(t)$ and possibly throwing out some curves, $\partial S_t$ has length at most $D g$, so $\partial S_t$ gives the require multicurve.

\begin{figure}[htbp]
 \centering
 \includegraphics[width=12cm]{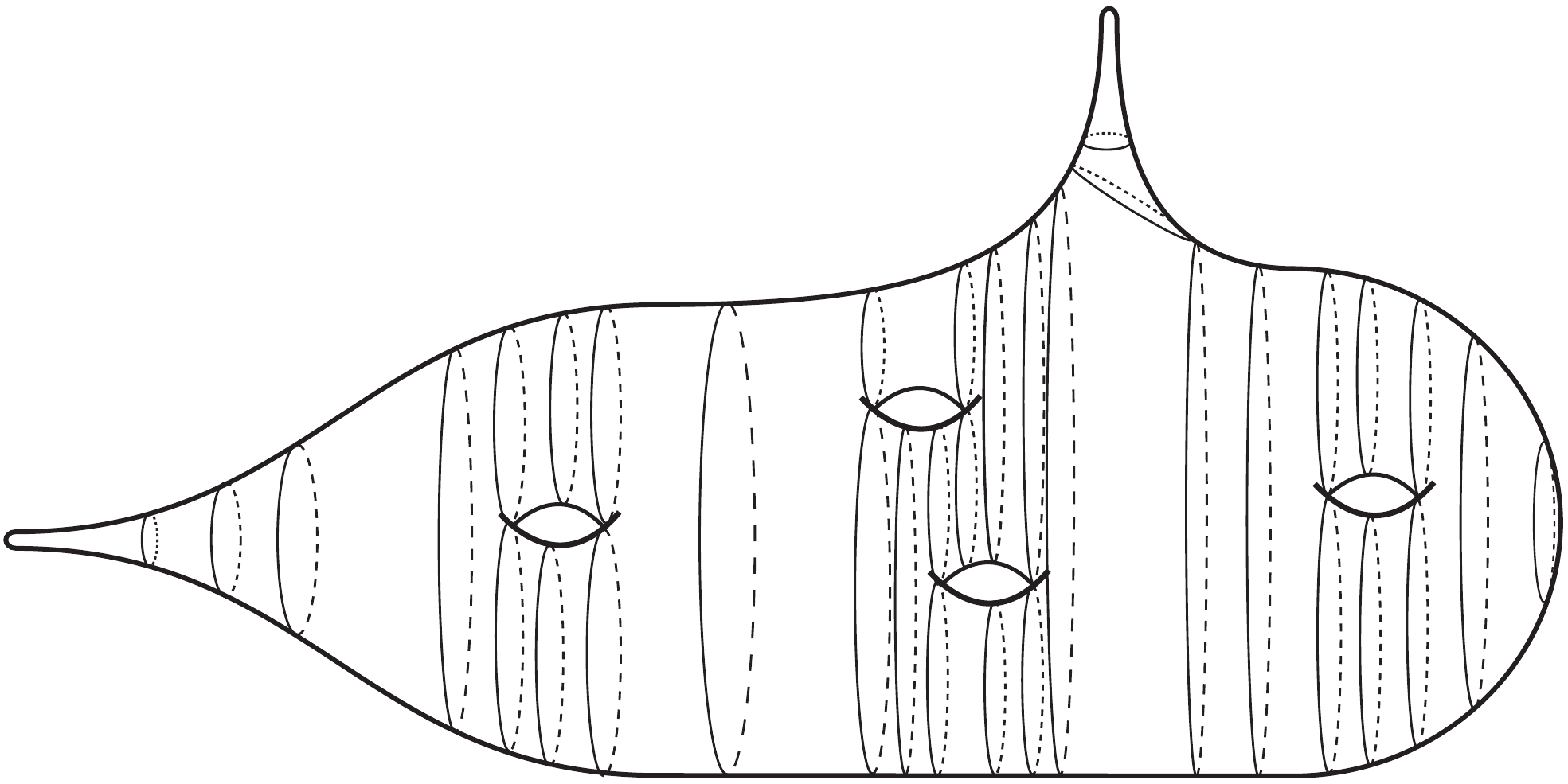}
 \caption{The level sets of $\phi$ sweep $S'$ from left to right}
 \label{fig:diastole}
\end{figure}

If $S$ is not closed, then we can approximate $S$ arbitrarily well by a closed surface $S'$ as follows. Consider the non-compact region in a cusp to bounded on one side by the horocycle of length $\varepsilon$. For very small $\varepsilon$, these regions are all disjoint and have area $\varepsilon$. One can now glue euclidean hemispheres of equator length $\varepsilon$ (thus of area $\frac{\varepsilon^2}{2\pi}$) to the resulting boundary curves. One obtains a closed surface of almost equal area and one can apply the argument given above to obtain a short multicurve. Note that by the collar lemma, for sufficiently small $\varepsilon$, the $1$-cycle given by tightening $\phi^{-1}(t)$ do not go deep into the cusps and therefore are not affected by the added hemispheres. The resulting short multicurve therefore gives a multicurve on $X$ with the desired property.
\end{proof}

Notice that in the previous lemma the subsurfaces $A$ and $B$ may not be connected. We can now proceed to the proof of lemma \ref{genuspinch}. 

\begin{proof}[Proof of Lemma \ref{genuspinch}] Given $n$, let $\mathcal{T}_n$ denote the set of all Teichm\"uller spaces of (possibly disconnected) finite area hyperbolic surfaces that can be built from at most $n$ pairs of pants, and let $F(n)$ denote the maximal distance from some point $X$ in some $T\in\mathcal{T}_n$ to a maximal node. Note that by definition $F(n)$ is monotonic. 

We claim that $F(n)$ satisfies an inequality of the form
\begin{equation}\label{eqn:fn}
 F(n)\le \sqrt{2}F(n/2+1)+ D n^{1/2}.
\end{equation}
To see this, let $X$ be an arbitrary point in $\mathcal{T}_n$. By Lemma \ref{hypdiast}, $X$ contains a multicurve $\mu$ of length at most $C\cdot\chi(X)$ such that $X\setminus\mu$ consists of two components whose areas differ by at most $\pi$. Note that this is true even if $X$ is disconnected. The time it takes to pinch this curve is at most $2\pi \sqrt{-C\chi(X)}=C'n^{1/2}$ by Theorem \ref{pinchlength}. The surface $X'$ given by pinching $\gamma$ has two (possibly disconnected) parts $A$ and $B$, where $n(A)$ and $n(B)$, the number of pairs of pants needed to build $A$ and $B$ respectively, are both at most $n/2+1$. The stratum $\mathcal{S}_\mu$ is a product with a product metric, so the distance from $X'$ to a maximal node is at most $\sqrt{F(n(A))^2+F(n(B))^2}\le\sqrt{2F(n/2+1)^2}$. This proves the claim. 

We now show that any function satisfying equation \ref{eqn:fn} is bounded above by $C\sqrt{n}\log(n)$ for some $C$. Assume such an inequality holds for all $m<n$, and replace $C$ by some constant larger than the constant $D$ given above if this is not already satisfied. 
\[
F(n)\le \sqrt{2}F(n/2+1)+ D n^{1/2}\le C \left(\sqrt{n/2+1}\log{(n/2+1)}+n^{1/2}\right)
\]
For $n>13$, $\left(\sqrt{n/2+1}\log{(n/2+1)}+n^{1/2}\right)\le \sqrt{n}\log(n)$. Choosing $C$ large enough to establish the base cases, Lemma \ref{genuspinch} holds by induction.

\end{proof}


\section{The Cubical Pants Graph}\label{sec:MCP}

In this section we introduce a metric space $\mathcal{CP}_{g,n}$, which we call the cubical pants graph. This metric space is designed to model the Weil-Petersson geometry of the set maximally noded surfaces in $\textrm{Teich}(\Sigma)$.

\subsection{Definition of the Cubical Pants Graph}

Recall the construction of the pants graph, introduced in the preliminaries above. Edges in the pants graph correspond to elementary moves, which occur in a subsurface homeomorphic to either a punctured torus or a 4-times punctured sphere. We will call two pants moves disjoint if they occur in disjoint subsurfaces. To construct $\mathcal{CP}_{g,n}$, we add an edge of length $\sqrt{k}$ between any points $P$ and $Q$ in the pants graph that differ by $k$ disjoint pants moves. The following lemma shows that distances in $\mathcal{CP}_{g,n}$ can be used to bound distances in Teichm\"uller space from above. Note that the standard pants graph graph distance does not provide a genus independent upper bound on distance between maximally noded surfaces. 

\begin{lemma}\label{cpgdominates} Let $S$ be a surface of genus $g$ with $n$ punctures, let $X$ and $Y$ be maximally noded surfaces in $\overline{\emph{Teich}(\Sigma)}$, and let $P$ and $Q$ be the corresponding pants decompositions. Then $\emph{d}_{WP}(X,Y)\le C \cdot \emph{d}_\mathcal{CP}(P,Q)$ for some constant $C$ independent of $g$ and $n$. 
\end{lemma}

\begin{proof}
Let $\gamma$ be a path in $\mathcal{CP}_{g,n}$, and let $P_1,P_2,\cdots,P_l$ be the list of vertices that $\gamma$ passes through.  Let $X_i$ be the maximally noded surface corresponding to $P_i$, let $g_i$ be the unique geodesic in $\overline{\textrm{Teich}(\Sigma)}$ from $X_i$ to $X_{i+1}$, and let $g$ be the piecewise geodesic path given by concatenating the $g_i$'s.  $P_i$ and $P_{i+1}$ differ by a collection of $k_i$ disjoint pants moves, so $X_i$ and $X_{i+1}$ share a multicurve $\mu$ of length zero that cuts both $X_i$ and $X_{i+1}$ into components which are homeomorphic to either 4-holed spheres or puncture tori. By the convexity of the strata, the geodesic between $X_i$ and $X_{i+1}$ stays in the stratum $\mathcal{S}_\mu$. $\mathcal{S}_\mu$ is a product of Teichm\"uller spaces of 4-holed spheres $\Sigma_{0,4}$ and once punctured tori $\Sigma_{1,1}$. $\overline{\textrm{Teich}(\Sigma_{0,4})}$ and $\overline{\textrm{Teich}(\Sigma_{1,1})}$ can both be modeled by $\{(x,y)\in\mathbb{R}^2~|~y\ge0~\textrm{or}~y=0~\textrm{and}~x\in\mathbb{Q} \}\cup{\infty}$, the upper half-plane union the rational points on the $X$ axis and the point at $\infty$. A pants move in either space correspond to traveling along the $y$-axis from $(0,0)$ to $\infty$, which has some finite length $C\in\overline{\textrm{Teich}(\Sigma_{1,1})}$ and length $C'\in\overline{\textrm{Teich}(\Sigma_{0,4})}$ (in fact $C'=2C$, for details see \cite{BM}). Using the product structure of the strata, we see that the distance from $X_i$ to $X_{i+1}$ is exactly $\sqrt{Ck+C'l}$, where $k$ is the number of punctured torus components of $\Sigma\setminus\mu$ and $l$ is the number of $4$-holed sphere components of $\Sigma\setminus\mu$. Since $k+l=k_i$, the number of moves by which $P_i$ and $P_{i+1}$ differ, we see that $\textrm{d}_{WP}(X_i,X_{i+1})\le \sqrt{\max(C,C') \,k_i}=\sqrt{\max(C,C')} \,\textrm{d}_{\mathcal{CP}}(P_i,P_{i+1})$. Thus $\textrm{d}_{WP}(X_1,X_l)$ is at most $\sqrt{\max(C,C')}\textrm{length}(\gamma)$. 
 
\end{proof}

We denote by $\mathcal{MCP}_{g,n}$ the space $\mathcal{CP}_{g,n}/\Mod(\Sigma_{g,n})$. We now give upper bounds on the diameter of $\mathcal{MCP}_{g,n}$ as $n$ and $g$ vary.  In this section it will be convenient to think of pants decompositions in $\mathcal{MCP}_{g,n}$ as graphs, by associating a vertex to each pair of pants and putting an edge between vertices whose corresponding pairs of pants share a curve. For convenience, we will add a vertex to the graph for each puncture and an edge connecting this vertex to the vertex corresponding to the pair of pants containing that puncture. We will refer to such terminal edges as {\it leaves}. Note that each vertex of the resulting graph will have valence either $1$ or $3$. 

It is an easy exercise show that there is a homeomorphism of $\Sigma$ carrying a pants decomposition $P$ to a pants decomposition $P'$ precisely when the corresponding graphs are isomorphic, so elements of $\mathcal{MCP}_{g,n}$ are in 1-1 correspondence with isomorphism types of connected graphs $G$ with $2g-2+n$ vertices where each vertex has valence either $1$ or $3$. It is easy to check that pants moves change the corresponding graph by one of the local moves shown in Figure \ref{moves}. These moves will be referred to as elementary moves.

\begin{figure}
 \centering
 \includegraphics[width=4 in]{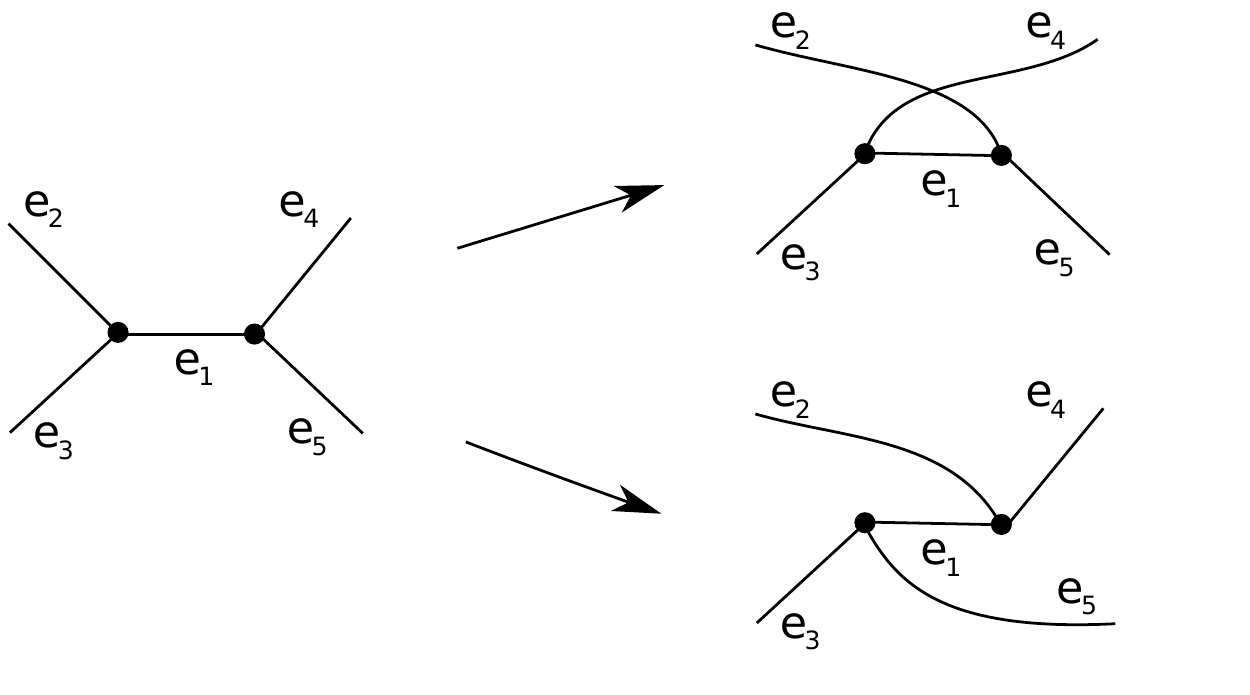}
 \caption{Transformations of a graph}
 \label{moves}
\end{figure}


\subsection{The diameter of $\mathcal{MCP}_{0,n}$}

In this subsection we prove an upper bound on the diameter of $\mathcal{MCP}_{0,n}$. Note that the vertices of $\mathcal{MCP}_{0,n}$, when considered as graphs, are connected binary trees.

\begin{prop}\label{mpsdiam} $\emph{diam}(\mathcal{MCP}_{0,n})\le C \sqrt{n}$ for some constant $C$.
\end{prop}

To show this, we provide an efficient algorithm that transforms any connected binary tree into the linear tree, i.e. the unique connected tree of maximal diameter.  In what follows, $e=2n-1$ will denote the number of edges. Given a tree $T$, we define a ``branch" of $T$ to be a component of $T\setminus v$ for some vertex $v\in T$.

\emph{Step 1: Trimming the tree}

Trimming is an inductive process that serves two purposes. The first is to consolidate the leaves of the tree in a way that ensures large scale branching, and the second is to partition the vertices into sets $V_k$ which will be used in step 2 below. 

To begin, we identify Y-branches of the tree $T$, i.e. branches containing exactly two leaves of $T$. We then look for all leaves in $T$ that are distance $1$ from each other and that do not belong to a Y-branch. We select a maximal collection of disjoint pairs of leaves from this collection. By performing a elementary move on an edge between two leaves, we may transform each of these pairs into a Y-branch to obtain a tree $T'$. Note that this transformation will cost $\sqrt{k}$ in the cubical pants graph if there are $k$ pairs involved. We will call a leaf of $T'$ isolated if it does not belong to a Y-branch.

\begin{lemma}\label{fewisolatedleaves} There are at most $\lfloor\frac{m}{2}\rfloor-1$ isolated leaves in $T'$.
\end{lemma}

\begin{proof}
Make $T'$ into a directed graph by picking a root vertex $v$ and directing every edge away from $v$. We will call a vertex an interior vertex if it is not the terminal end of a leaf, and denote the set of such vertices $\mathcal{I}$. Let $A\subset \mathcal{I}$ denote the set of interior vertices that are the base of an isolated leaf. Given $x\in A$, any neighboring interior vertex is not the base of an isolated leaf as otherwise it would have been paired off into a Y-branch. Since $x$ is not a part of a $Y$-branch by assumption, there is directed edge from $x$ to another interior vertex $y$. By arbitrarily picking one such $y$ for each base $x\in \mathcal{I}$, we get an injective map from $A$ to $\mathcal{I}\setminus A$. Thus $|A|\le \lfloor |\mathcal{I}|/2\rfloor$. Since $|\mathcal{I}|= m-2$ this establishes Lemma \ref{fewisolatedleaves}. 
\end{proof}

We now form a set $V_0$ consisting of vertices at the center of Y-branches. We then delete the Y-branches and replace them by leaves to form a new tree $T_1$. 

Given a tree $T_i$, we repeat the above process to form a set $V_i$ and tree $T_{i+1}$ until reaching a final tree $T_M$ has either one or two interior vertices remaining. For simplicity in what follows, we assume $T_M$ has one interior vertex, since if two vertices remain we pick one of them arbitrarily (which will necessarily be the center of a Y-branch) and trim it off. We will denote the final vertex by $v_0$. 

We can reconstruct a tree with the same number of leaves as $T$ by replacing the trimmed off Y-branches (but not undoing the elementary moves). The resulting tree $S$ has a base vertex and looks (almost) like a binary tree branching off of $v_0$. This process is illustrated in figure \ref{trimalg}.

\begin{figure}[htbp]
 \centering
 \includegraphics[width=6 in]{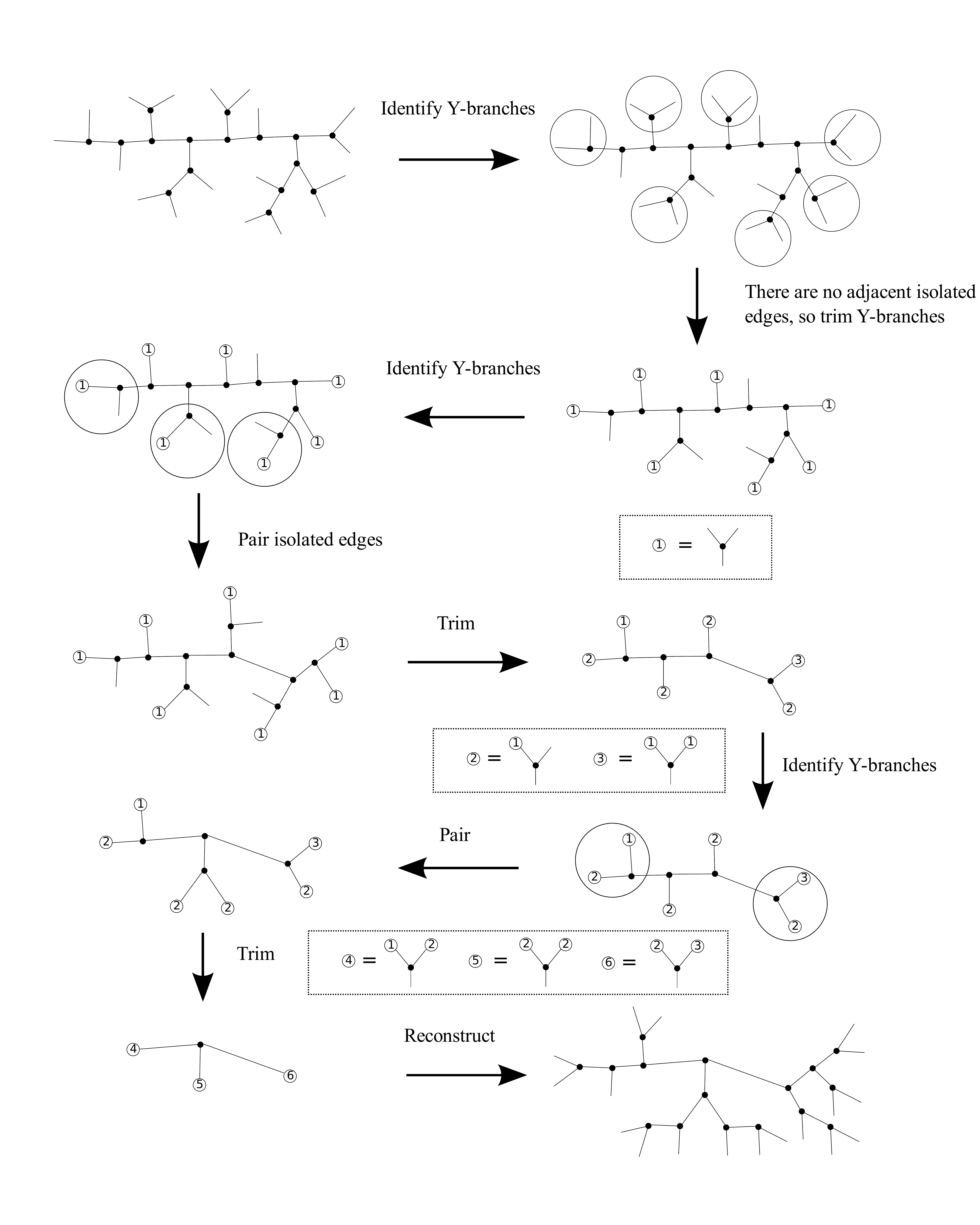}
 \caption{Trimming a tree}
 \label{trimalg}
\end{figure}

\begin{lemma}\label{trimminglength} The distance between $T$ and $S$ is at most $C\sqrt{n}$ for some $C$.
\end{lemma}

\begin{proof}

Let $m_i$ be the number of leaves of $T_i$. By Lemma \ref{fewisolatedleaves}, fewer than half of the leaves of $T_i$ are isolated. Since each pair of non-isolated leaves is replaced by a single leaf in $T_{i+1}$, this means that $m_{i+1}\le m_i-1/4m_i=3/4m_i$. We also note that the number of Y-branches is at most $m_i/2$, so $m_i/2$ bounds the number of simultaneous moves used to get from $T_i$ to $T_{i+1}$ and $|V_i|$, the cardinality of $V_i$. This shows that the trimming algorithm must end in at most $M$ steps, where $M\le \log_{4/3}(n)$. 

Let $T_i^c$ denote the tree formed by gluing back on the Y-branches of $T_i$. Since $T_i$ is transformed into $T_{i+1}$ by at most $m_i/2$ disjoint moves, $T_i^c$ and $T_{i+1}^c$ are distance at most $\sqrt{m_i/2}$ apart. Furthermore, since $m_{i+1}\le 3/4 m_i$,
\[
\sqrt{m_0\over 2}+\sqrt{m_1\over 2}+\cdots+\sqrt{m_M\over 2}\le
\sqrt{n\over 2}\sum_{i=1}^M\left(\sqrt{3\over 4}\right)^{i-1}
\le C\sqrt{n}
\]

\end{proof}

\emph{Step 2: Melting the tree to the linear graph}

We now transform $S$ into the linear graph. We begin by considering a line $L_0$ that contains $v_0$. By a line, we mean a maximal connected non-branching subtree. We shall use the decomposition of the set of vertices given by $V_M\sqcup V_{M-1}\sqcup\cdots\sqcup V_1$ obtained in the trimming step above. Note that each vertex in $V_k$ has an edge to a vertex in $V_{k+l}$ for some $l>0$.

Suppose by induction that we have a tree $S_k$ containing a line $L_k$ such that each vertex from $V_{M-k}$ is contained in $L_k$. Since each vertex $v$ in $V_{M-k-1}$ in connected by an edge $e_v$ to a vertex $v'$ that is already in $L_k$ and $v$ is a branching vertex, we can perform a elementary move brings $e_v$ into the line $L_k$. Since all the edges we are performing moves on are disjoint, we can turn $T_k$ into a new tree $T_{k+1}$ containing a line $L_{k+1}$ which contains all the edges $e_v$. The distance traveled in the cubical pants graph in doing this is $\sqrt{|V_{M-k-1}|}$. 

By iterating this process, we eventually end up with a tree of maximal diameter. Since the size of $V_{M-k}$ is at most 3/4 the size of $V_{M-k}$ as remarked before, the distance from $S$ to the linear graph is at most
\[
\sum_{k=1}^M \sqrt{|V_{M-k}|} \le \sqrt{n\over2}\sum_{k=1}^M \left({3\over4}\right)^k\le C\sqrt{n}
\]

Since given any two trees $T_1$ and $T_2$ the distance from $T_i$ to the linear graph is at most $C\sqrt{n}$ for some $C$, the distance between $T_1$ and $T_2$ is at most $2C\sqrt{n}$ so Proposition \ref{mpsdiam} follows. 

As an immediate corollary to Proposition \ref{mpsdiam} and Lemma \ref{cpgdominates} we obtain the upper bound on the diameter of moduli space of $n$-times punctured spheres.

\begin{cor} $\emph{diam}(\mathcal{M}_{0,n})\le D \sqrt{n}$ for some $D$.
\end{cor}


\subsection{The diameter of $\mathcal{MCP}_{g,0}$}

We now prove the analogous proposition for the $\mathcal{MCP}_{g,0}$.

\begin{prop}\label{goalprop} There exists a constant $C$ such that $\emph{diam}(\mathcal{MCP}_{g,0})\le C \sqrt{g}\log(g)$. 
\end{prop}

We will call $P\in \mathcal{CP}_{g,0}$ {\it treelike} if $P$ has $g$ loops of length $1$. Note that it is an immediate corollary of our bound on $\textrm{diam}(\mathcal{MCP}_{0,n})$ that any two treelike pants decomposition graphs are distance at most $C' \sqrt{g}$ apart, since we can delete the loops, apply the above algorithm and then put the loops back. To show Proposition \ref{goalprop}, it thus suffices to show the following.

\begin{prop}\label{mainprop} There exists a constant $C$ such that for any $P\in \fpg$ there exists a treelike pants decomposition $T\in \fpg$ with $d_{\fpg} (P,T) < C \sqrt{g} \log g$.
\end{prop}

To show this we will need the following simple graph-theoretic lemmas. The first gives an upper bound on the systole, or shortest non-trivial cycle, in a pants decomposition graph.

\begin{lemma}
$\sys(P) < \log_2 (g)$
\end{lemma}

\begin{proof} There are $2g-2$ vertices in the graph. Let $v\in P$. If $r< \sys(P)$ then the $r$-ball around $v$ is a 3-regular tree and has $3\cdot 2^r-2$ vertices, so $3\cdot 2^r-2<2g-2$ and hence $r<\log_2(g)$. 
\end{proof}

We will call a pair of embedded cycles in a graph disjoint if they do not share any edges. Note that in the following lemma we allow $P$ to be the pants decomposition of surface with punctures.

\begin{lemma}\label{disjointcycles}
Any $P\in \fpgn$ contains at least $A\frac{g}{\log g}$ pairwise disjoint embedded cycles where $A$ can be taken to be $\frac{\log 2}{2}$.
\end{lemma}

\begin{proof}
We first observe that given $P\in \fpgn$ we can find $P'\in \mathcal{CP}_{g,0}$ such that then number of disjoint cycles on $P'$ is the same as the number of disjoint cycles on $P$. We form $P'$ from $P$ simply by deleting leaves along with their attaching vertices until no leaves remain. Observe that a collection of disjoint cycles on $P'$ give rise to a set of disjoint cycles on $P$ by replacing the deleted leaves and their attaching vertices, so the number of the disjoint cycles is not not affected.

We now begin with $P$ of genus $g$ with $n$ punctures, and argue by induction on genus. For the purpose of our proof, it will be convenient to allow $P$ to be disconnected. Note that for $A={\log 2 \over 2}$ the lemma holds for $g=2$.

Now consider $P$ of genus $g\geq 3$, and let $P'$ be the graph obtained by removing punctures. The previous lemma tells us that there is cycle $\gamma'$ on $P'$ of length at most $A^{-1} \log g$ (note that this holds even if $P$ is disconnected).  As mentioned above, $\gamma$' gives rise to a cycle $\gamma$ on $P$. Note that $P'\setminus \gamma'$ has genus at least $g - \ell(\gamma)\geq g - A^{-1} \log g$, and thus so does $P\setminus \gamma$.

By induction, $P\setminus \gamma$ contains at least $\frac{g - A^{-1} \log g}{ \log( g - A^{-1} \log g ) }$ disjoint cycles. Thus $P$ contains at least
\[
1 + \frac{g - A^{-1}\log g}{A^{-1} \log( g - A^{-1} \log g ) } > 1 + \frac{g - A^{-1} \log g}{A^{-1} \log g } = A\frac{g}{\log g}
\]
cycles. 
\end{proof}

Note that the above estimate is roughly optimal. Indeed, there exists a sequence of 3-regular graphs $P_g$ such that $\sys(P_g)\ge U\log g$ for some constant $U$. Since every cycle in such a graph has length $U \log g$, it is impossible to hope for more than $\frac{3g-3}{ U \log g}$ disjoint cycles.

We now describe an algorithm that takes an arbitrary $P\in\fpgn$ to a treelike graph. The general strategy is as follows. Given a collection of disjoint cycles $\gamma_1, \dots,\gamma_k$, we can do simultaneous moves to reduce each of their lengths. Each collection of simultaneous moves reduces the length of $\gamma_i$ by at least $l(\gamma_i)/2$, so in a bounded number of steps each cycle can be transformed into a cycle of length 1. We call this the ``genus reduction step," since the resulting graph looks like a graph of lower genus with some additional loops. We then move these loops into a treelike subgraph (the ``loop sorting step"). By removing this treelike subgraph, we end up with a graph of lower genus. We repeat the process until each embedded cycle in the graph is a loop.

\emph{Step 1: Genus reduction}

The goal of this step is to transform any disjoint set of cycles in $P\in\fpgn $ to a collection of length 1 loops in roughly $\sqrt{g}$ moves.

Given a single cycle $\gamma$, we can reduce its length as follows. We choose a maximal set of disjoint edges of $\gamma$, which will contain $\lfloor l(\gamma)/2\rfloor$ edges. By simultaneously performing elementary moves on these edges, we can remove these edges from the cycle thereby shortening it to have length $l - \lfloor l(\gamma)/2\rfloor$. This process is repeated until the cycle has length $1$. Note that since the edges are disjoint, the cost of each shortening step in $\fpgn$ is at most $\sqrt{\lfloor l(\gamma)/2\rfloor}$. 

We can also apply the above procedure to any collection of disjoint cycles, since cycles remain disjoint under the length reducing transformations. Note that the length of any cycle in the collection is reduced to at most $\frac{2}{3}$ its original length each step unless the loop has length 1 (in fact the length of a cycle $\gamma$ becomes $l(\gamma)-\lfloor l(\gamma)/2 \rfloor$, but we use the value 2/3 for simplicity). Since the total length of the collection of cycles is reduced by a factor of $2/3$, the number of simultaneous elementary moves that is done in each step of the process decreases by a factor of $\frac{2}{3}$ as well. At the initial step there are at most $(3g-3)/2$ simultaneous moves, so the total distance traveled to reduce all cycles to length 1 cycles is bounded by

\[
\sqrt{\frac{3g-3}{2}} \sum_{k=1}^{\infty} \left(\frac{2}{3}\right)^k \le C\sqrt{g}.
\]


\emph{Step 2: Sorting the loops}

In this step we transform our graph $G$ into a graph $G'$ which contains an edge $e$ such that \\$G'\setminus\{e\}=A\sqcup B$ where $A$ is a treelike graph and $B$ contains no loops. 

Let $T$ be a spanning tree of $G$. We will refer to the leaves of $T$ that touch loops of $G$ as ``loop" leaves. By the results of the previous subsection, any tree is distance at most $C \sqrt{n}$ away from a linear tree, where $n$ is the number of leaves. Since $T$ has at most $2 g$ leaves we can transform the subgraph $T$ into a linear subgraph $L$ by traveling distance at most $C' \sqrt{n}$ in $\fpgn$.

We now partition $L$ into connected segments with $4$ vertices each. Four leaves attach to each segment, each of which is either a loop leaf or a non-loop leaf. If we do not see two consecutive leaves of the same type in a given segment, then we see an alternating segment, as shown in figure \ref{altsequence}. We can then perform a elementary move on the middle edge of such a segment to bring loops of the same type together. We do this to each segment of $L$ to produces a tree $T'$, which is at most $\sqrt{g/2}$ .

\begin{figure}[htbp]
 \centering
 \includegraphics[width=3 in]{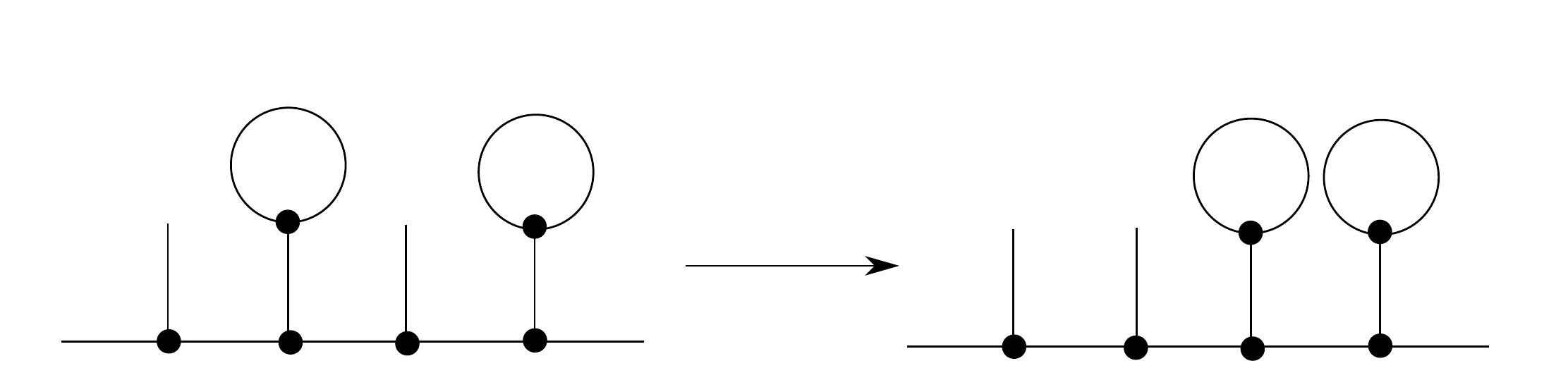}
 \caption{Sorting loops}
 \label{altsequence}
\end{figure}

In each segment we can split off a pair of leaves of the same type as shown in figure \ref{contract}, again at cost at most $\sqrt{g/2}$.

\begin{figure}[htbp]
 \centering
 \includegraphics[width=3 in]{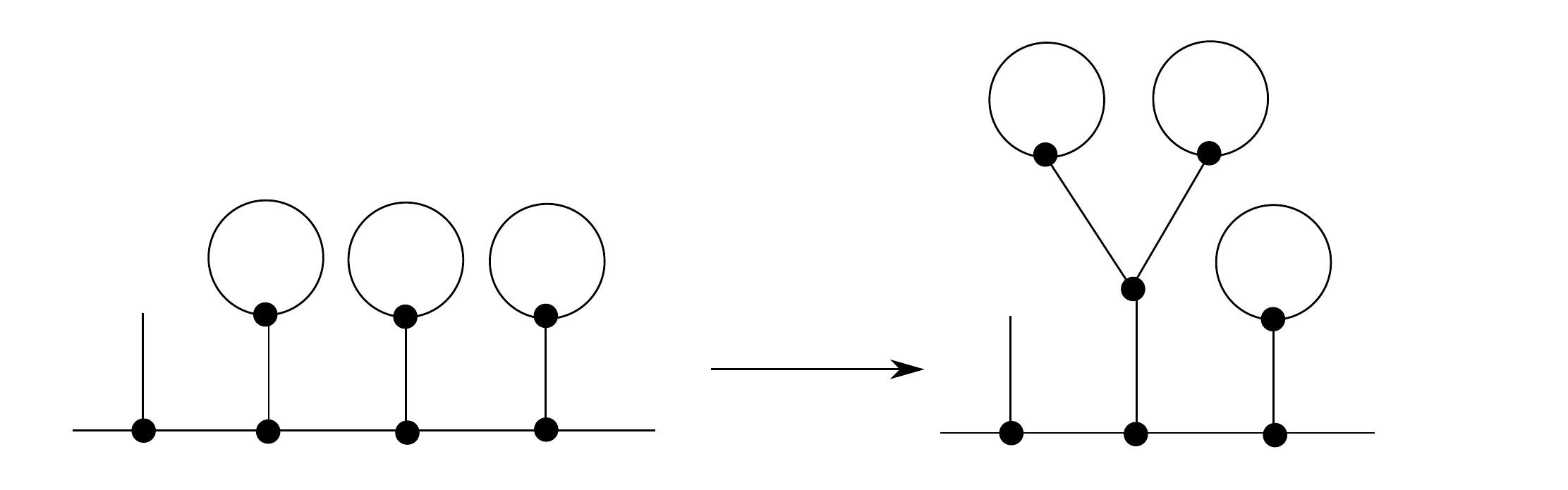}
 \caption{Shortening the line}
 \label{contract}
\end{figure}

Note that we now have a new line $L_1\subset T'$ given by cutting off the loops leaves and Y-branches of $L$, and the number of vertices of $L_1$ is roughly $\frac{3}{4}$ the number of vertices of $L$. The edges coming off of $L_1$ are again of two types: those that are branches $T'$ with only loop leaves, and those that are branches with only non-loop leaves. We can apply the above procedure to $L_1$, to consolidate branches of loops and branches of non-loops. Iterating this procedure, we always have a base line $L_k$ such that each edge attached to $L_k$ is the base of a branch whose leaves are all of loop type, or all of non-loop type. This algorithm terminates when all the loop leaves have been consolidated into a single branch. Since $|L_k|$, the number of vertices in $L_k$, is bounded above by $3/4|L_{k+1}|$, the distance traveled during this step is at most $C\sqrt{g}$ for some $C$.


\emph{Step 3: Removing the tail and iteration}

After performing the steps 1 and 2, we have transformed $P=P_0$ into a graph $P_1$ which has a connected subgraph that is a treelike collection of loops. We call this subgraph the tail of $P_1$. By cutting off this tail and replacing it with a single loop leaf, we end up with a graph $P_1'$ with genus $g'$ which we will call the effective genus of $P_1$. We can now perform steps 1 and 2 to $P_1'$ (without destroying this special loop). This will move us by at most $C\sqrt{g'}$. By putting the tail back in place of the replacement loop, we end up with a graph $P_2$.

Repeating this process we end up with a sequence $P=P_0,P_1,\dots,P_N$, where $P_N$ is a treelike graph. The distance between $P_i$ and $P_{i+1}$ is bounded above by $C\sqrt{g_i}$, where $g_i$ is the effective genus of $P_i$. The following lemma proves Proposition \ref{mainprop}.

\begin{lemma} The distance between $P_0$ and $P_N$ is at most $C \sqrt{g}\log(g)$ for some $C$.
\end{lemma} 

\begin{proof}

We have that this distance is at most $C \sum_{i=0}^N \sqrt{g_i}$. By Lemma \ref{disjointcycles}, $g_i-g_{i+1}\ge A g_i /\log(g_i)$, so $g_{i+1}\le(1-A/\log(g_i))$. Let $i_1$ be the first index such that $g_{i_1}\le g/2$. For all $k<I_1$, the k-th step reduces the genus by
\[
\frac{Ag_i}{\log(g_i)}\le \frac{A g/2}{\log(g/2)}
\]
so at most $\log(g)/A$ moves can occur before the genus falls below $g/2$, and hence $i_1\le \log(g)/A$. We define $i_k$ similarly to be the first index such that $g_{i_k}\le g/2^k$, and by the above argument $i_k-i_{k-1}\le \log(g/2^{k-1})/A$.

\[
 \sum_{i=1}^N \sqrt{g_i} \le \sum_{k} (i_{k+1}-i_{k})g_{i_k}
\le\frac{1}{A} \sum_k \sqrt{\frac{g}{2^k}}\log\left(\frac{g}{2^{k-1}}\right)\le \frac{\sqrt{g}\log(g)}{A}\sum_k 2^{-k/2}<C \sqrt{g}\log(g).
 \]
 
\end{proof}

As a corollary we obtain the upper bound from Theorem \ref{thm:ggrowth} in the case where the underlying surface has no punctures. 

\begin{cor} $\emph{diam}(\mathcal{M}_{g,0})\le C \sqrt{g}\log(g)$ for some $C$.
\end{cor}

\subsection{Lower Bounds on the Diameter of $\mathcal{MCP}_{g,n}$}

In this section we prove lower bounds on the diameters of the combinatorial models.  Note that these results do not provide lower bounds on the Weil-Petersson diameter of $\mathcal{M}_{g,n}$, since it is not known whether the quasi-isometric embedding $\mathcal{CP}_{g,n}\hookrightarrow\textrm{Teich}(\Sigma)$ has quasi-isometry constants that are independent of $g$ and $n$.  Both lower bounds in this section are easily derived from lower bounds on the diameter of $\mathcal{P}_{g,n}/\textrm{Mod}(\Sigma)$.

\begin{prop}  There exists $c$ such that $\emph{diam}(\mathcal{MCP}_{0,n})\ge c \sqrt{n}$.
\end{prop} 

\begin{proof}  Let $P_1$ be a pants decomposition in $\mathcal{P}_{0,n}$ whose underlying cubic graph has maximal diameter and let $P_2$ be a pants decomposition whose underlying graph has minimal diameter.  Note that the diameter $d_1$ of $P_1$ is $n-3$, and the diameter $d_2$ of $P_2$ is less than $\log_2(n)$.  The distance in $P_{0,n}$ between $P_1$ and $P_2$ is therefore at least $n-\log_2(n)+3$, since the diameter of the underlying cubic graph changes by at most one per pants move.  Let $(P_1=p_1,p_2,\dots,p_k=P_2)$ be a geodesic between $P_1$ and $P_2$ in $\mathcal{CP}_{0,n}$.  We have

\[
d_{\mathcal{CP}_{0,n}}(P_1,P_2)= \sum_{i=1}^{k-1}\left(d_{\mathcal{P}_{0,n}}(p_i,p_{i+1})\right)^{1/2}\ge \left(\sum_{i=1}^{k-1}d_{\mathcal{P}_{0,n}}(p_i,p_{i+1})\right)^{1/2}\ge (n-\log_2(n)+3)^{1/2}
\]
 
\end{proof}

In the case of closed surfaces, a lower bound on $\mathcal{P}_{g,0}$ can be produced by a counting argument.  The following lemma was observed by Thurston.

\begin{lemma}\label{lem:pantsdiameter} There exists $c$ such that $\emph{diam}(\mathcal{P}_{g,0})\ge c g\log(g)$.
\end{lemma}

\begin{proof}
Theorem 2.3 in \cite{STT} shows that the number of points in a ball of radius $r$ in $\mathcal{P}_{g,0}$ is at most $3^{(2g-2)+3\cdot r}$.  It is known (see e.g. \cite{RW}) that the total number of cubic graphs with $2g-2$ vertices is at least $(a g)^g$ for some constant $a$.  Thus if $r= \textrm{diam}(\mathcal{P}_{g,0})$, then $3^{(2g-2)+3\cdot r}\ge (ag)^g$, so 
\[
(2g-2+3\cdot r)\log(3)\ge g\log (g)+ g\log (a)
\]
from which the result follows.    

\end{proof}

\begin{prop}  There exists $c$ such that $\emph{diam}(\mathcal{MCP}_{g,0})\ge c\sqrt{g}\log(g)$.
\end{prop} 

\begin{proof}  By Lemma \ref{lem:pantsdiameter}, there exist pants decompositions $P_1$ and $P_2$ in $\mathcal{P}_{g,0}$ at distance at least $c g\log(g)$.  Let $(P_1=p_1,p_2,\dots,p_k=P_2)$ be the pants decompositions along a geodesic in $\mathcal{CP}_{g,0}$ from $P_1$ to $P_2$.  We have that 

\[
d_{\mathcal{CP}_{g,0}}(P_1,P_2)= \sum_{i=1}^{k-1}\left(d_{\mathcal{P}_{g,0}}(p_i,p_{i+1})\right)^{1/2}.
\]
Since $p_i$ can be transformed into $p_{i+1}$ by simultaneous moves, $d_{\mathcal{P}_{g,0}}(p_i,p_{i+1})$ is at most $g$. Given any collection $\mathcal{A}$ of numbers bounded above by $g$ whose sum equals $N$, it is elementary to show that the $\sum_{a\in\mathcal{A}} \sqrt{a}\ge (N/g)\sqrt{g}$.  Since $\sum_{i=1}^{k-1}d_{\mathcal{P}_{0,n}}(p_i,p_{i+1})\ge c g\log(g)$, it follows that

\[
d_{\mathcal{CP}_{g,0}}(P_1,P_2)\ge (c g\log(g)/g)\sqrt{g}=c \sqrt{g}\log(g)
\]

\end{proof}


\section{Lower Bounds on $\textrm{diam}(\mathcal{M}_{g,n})$}\label{sec:LB}

This section establishes lower bounds on the diameter of moduli space in the case of punctured spheres and closed surfaces.  A lower bound in terms of the Euler characteristic of the underlying surface can be derived as a simple consequence of the geodesic convexity of the strata.  

\begin{prop} There exists a constant $c$ such that $c\le\emph{diam}(\mathcal{M}_{g,n})/\sqrt{|\chi(\Sigma_{g,n})|}$
\end{prop}

\begin{proof}
Let $F(n)=\textrm{diam}(\mathcal{M}_{0,n})$. It is clear that $M_{0,n-1}$ sits as a factor of a stratum of $M_{0,n}$, so $F(n)$ is monotonic by the convexity of the strata. By pinching off a 4-holed sphere and using the product structure of the strata, we see that $F(n+2)\ge \sqrt{F(n)^2+b^2}$, where $b$ is the diameter of $\mathcal{M}_{0,4}$. This shows that $F(2n)^2\ge b^2(n-1)$, and that $F(n)\ge (b~/\sqrt{2})\sqrt{n-4}$.  The case of $g>0$ reduces easily to planar case by observing that $\mathcal{M}_{0,2g+n}$ sits as a substratum of $\mathcal{M}_{g,n}$.

\end{proof}

The lower bound for $\textrm{diam}(\mathcal{M}_{g,0})$ we give here is based on the Bishop-Gromov inequality. Let $\mathcal{M}_{g,0}^{\varepsilon}$ denote the $\varepsilon$-thin part of moduli space, i.e. the set of $X\in\mathcal{M}_{g,0}$ such that $X$ has a closed geodesic $\gamma$ such that $l(\gamma)\le\varepsilon$. Let $\mathcal{M}_{g,0}^{T(\varepsilon)}$ denote the complement of this set, the $\varepsilon$-thick part of moduli space.  

The  $\varepsilon$-thick part of moduli space is not geodesically convex, however by the geodesic convexity of interior of moduli space there is some strictly positive lower bound on the systole function along any geodesic between two points in $\mathcal{M}^{T(\varepsilon)}_{g,0}$.  Given a point $p\in\mathcal{M}^{T(\varepsilon)}_{g,0}$, let $\tilde{p}$ a preimage of $p$ in Teichm\"uller space.  Let $D(\tilde{p})$ be a Dirichlet fundamental domain for $\mathcal{M}_{g,0}^{T(\varepsilon)}$ based at $\tilde{p}$, i.e.

\[
D(\tilde{p})=\{x\in \textrm{Teich}(\Sigma_{g,0})~|~\forall g\in \textrm{Mod}(\Sigma_{g,0}) , d(\tilde{p},x)\le d(\tilde{p},g\cdot x) \}.
\]
Let $D(\tilde{p},\varepsilon)$ denote the intersection of $D(\tilde{p})$ with the $\varepsilon$-thick part of Teichm\"uller space, and let $R(\tilde{p},\varepsilon)=\cup_{x\in D(\tilde{p},\varepsilon)}[\tilde{p},x]$ where $[\tilde{p},x]$ is the unique geodesic in Teichm\"uller space between $\tilde{p}$ and $x$.  We will call $R(\tilde{p},\varepsilon)$ the {\it radial closure} of $\mathcal{M}^{T(\varepsilon)}_{g,0}$ from $\tilde{p}$.  We define $v_{\varepsilon}(g)$ by
\[
v_{\varepsilon}(g)=\sup_{p\in \mathcal{M}^{T(\varepsilon)}_{g,0}} \inf_{x\in R(\tilde{p},\varepsilon)} \textrm{sys}(x),
\]
where $\textrm{sys}(x)$ denotes the systole of the surface $x$.  Note that $v_\varepsilon(g)=\varepsilon$ if $\mathcal{M}^{T(\varepsilon)}_{g,0}$ has a star-shaped fundamental domain.

\begin{prop} There exists a constant $c$ such that moduli space of closed surfaces $\mathcal{M}_{g,0}$ satisfies
\[
c \le \liminf_{g\to\infty}\frac{\emph{diam}(\mathcal{M}_{g,0})}{v_{\varepsilon}(g)\sqrt{g}\log{g}}
\]
\end{prop}

\begin{proof}

The Bishop-Gromov inequality gives that if $M^N$ is a Riemannian manifold of dimension $N$ with Ricci curvature bounded below by $(n-1)k$, then the volume of any star-shaped domain with radii of length less than $R$ in $M^N$ is bounded above by the volume of a ball of radius $R$ in $S_{k}^N$, the simply connected $N$-manifold of constant curvature $k$.  

There exists a point $p$ realizing the supremum taken in the definition of $v_{\varepsilon}(g)$ by the compactness of $\mathcal{M}^{T(\varepsilon)}_{g,0}$.  Let $p$ be such a point and let $R(\tilde{p},\varepsilon)$ be the radial closure of $\mathcal{M}^{T(\varepsilon)}_{g,0}$ from some lift $\tilde{p}$ of $p$.  Note that by definition, $R(\tilde{p},\varepsilon)$ is a star-shaped domain with center $\tilde{p}$.  Let $D(g)=\sup_{x\in R(\tilde{p},\varepsilon)} d(\tilde{p},x)$.  Since geodesics contained within a Dirichlet fundamental domain are minimizing, $D(g)\le \textrm{diam}(\mathcal{M}_{g,0})$. 

We have that $\textrm{Vol}\left(\mathcal{M}_{g,0}^{T(\varepsilon)}\right)\le \textrm{Vol}\left(R(\tilde{p},\varepsilon)\right)$.  Let $c$ be an arbitrary constant in the interval $(0,1)$ and let $\lambda=1-c$. By Mirzakhani's theorem, Theorem \ref{thm:mirz}, there exists a uniform $\varepsilon$ such that $\textrm{Vol}(\mathcal{M}_{g,0}^{\varepsilon})<c \textrm{Vol}(\mathcal{M}_{g,0}),$ so

\[
\textrm{Vol}(R(\tilde{p},\varepsilon))\ge \textrm{Vol}(\mathcal{M}_{g,0}^{T(\varepsilon)})\ge \lambda \textrm{Vol}(\mathcal{M}_{g,0}), 
\]

Teo's curvature bound, Theorem \ref{thm:teo},  gives that for small enough $\varepsilon$, $\textrm{Ric}(R(\tilde{p},\varepsilon))\ge -v_\varepsilon(g)^{-2}$. Let $\alpha=v_\varepsilon(g)^{-2}$. Since for all $x\in T(\tilde{p},\varepsilon)$, $\textrm{Ric}(x)\ge\frac{-\alpha}{6g-7}(6g-7)$, we can use the Bishop-Gromov inequality to compare the volume of $R(\tilde{p},\varepsilon)$ to the volume of a ball of radius $D(g)$ in $S_{-k}^{~N}$, where $k=\frac{\alpha}{N-1}$ and $N= 6g-6$. 

\[
\textrm{Vol}\left(\mathcal{M}_{g,0}^{T(\varepsilon)}\right)\le\textrm{Vol}\left(R(\tilde{p},\varepsilon)\right)\le \textrm{Vol}\left(B^{S^{~N}_{-k}}_D(0)\right)=
\]
\[
\textrm{Vol}_{\textrm{Euc}}\left(S^{N-1}\right)\left(\sqrt{k}\right)^{1-N}\int_0^{D(g)} \sinh\left(\sqrt{k}t\right)^{N-1}\textrm{dt},
\]

where $\textrm{Vol}_{\textrm{Euc}}(S^{N-1})$ is the standard $(N-1)$-dimensional volume of the $(N-1)$-sphere.

\[
\textrm{Vol}_{\textrm{Euc}}\left(S^{N-1}\right)\left(\sqrt{k}\right)^{1-N}=\frac{2\pi^{N/2}}{\Gamma(N/2)}\left(\frac{N-1}{\alpha}\right)^{(N-1)/2}\le 2\pi^{N/2}\left(\frac{e}{N/2-1}\right)^{N/2-1}\left(\frac{N-1}{\alpha}\right)^{(N-1)/2}
\]

by Stirlings approximation. It is easy to show that this expression is bounded above by $C^N$ for some $C$. Thus

\[
\textrm{Vol}\left(\mathcal{M}_{g,0}^{T(\varepsilon)}\right)\le C^N\int_0^{D(g)} \sinh(\sqrt{k}t)^{N-1}\textrm{dt}
\le
C^N\int_0^{D(g)} e^{\sqrt{k}(N-1)t}\textrm{dt}\le C^N e^{\sqrt{\alpha (N-1)}D(g)}
\]

By our choice of $\varepsilon$, $\textrm{Vol}\left(\mathcal{M}_{g,0}^{T(\varepsilon)}\right)\ge \lambda \textrm{Vol}(\mathcal{M}_{g,0})$, so we have the inequality $\lambda \textrm{Vol}(\mathcal{M}_{g,0})\le C^Ne^{\sqrt{\alpha (N-1)}D(g)}$. Taking logarithms and dividing by $g\log{g}$ we get

\[
\frac{\log \textrm{Vol}(\mathcal{M}_{g,0})}{g \log g}+\frac{\log{\lambda}}{g\log{(g)}}\le
\frac{1}{g\log g}\left(\sqrt{\alpha(N-1)}D(g)+N\log{C}\right).
\]

Taking the limit as $g$ goes to infinity, and applying Schumacher and Trapani's theorem, Theorem \ref{thm:schu}, we get

\[
2\le \liminf_{g\to\infty} \frac{\sqrt{6\alpha}D(g)}{\sqrt{g} \log{g}}=  \liminf_{g\to\infty} \frac{\sqrt{6}D(g)}{v_{\varepsilon}(g)\sqrt{g} \log{g}}.
\]
Since $D(g)\le \textrm{diam}(\mathcal{M}_{g,0})$, this establishes the claim.
\end{proof}


\section{Behavior of $\textrm{diam}(\mathcal{M}_{g,n})$ in the Limit}\label{sec:GS}

In this section we show that the asymptotic bounds on diameter as $g$ varies do not depend on $n$, and likewise that the asymptotic bounds in $n$ do not depend on $g$.

\begin{lemma}  Suppose that $f(g)$ and $h(g)$ are functions such that 
$\limsup_{g\to\infty} \emph{diam}(\mathcal{M}_{g,0})/f(g)\le C$
and $\liminf_{g\to\infty} \emph{diam}(\mathcal{M}_{g,0})/h(g)\ge c.$  Then $\limsup_{g\to\infty} \emph{diam}(\mathcal{M}_{g,n})/f(g)\le C$ and $\liminf_{g\to\infty} \emph{diam}(\mathcal{M}_{g,n})/h(g)\ge c$.
\end{lemma}

\begin{proof} This lemma is a simple consequence of the convexity of the strata. For the result concerning the upper bound, note that $\textrm{diam}(\mathcal{M}_{g,n})\le\textrm{diam}(\mathcal{M}_{g+n,0})$. This is because $\mathcal{M}_{g,n}$ sits as a factor of a stratum of $\mathcal{M}_{g+n,0}$, since we can pinch $n$ curves having a 1-holed torus as a complementary component. Thus 
\[
\lim_{g\to\infty}\frac{\textrm{diam}(\mathcal{M}_{g,n})}{f(g)}\le
\lim_{g\to\infty}\frac{\textrm{diam}(\mathcal{M}_{g+n,0})}{f(g)}\le C
\]
For the lower bound, suppose first that $n\ge 2$. Since $\mathcal{M}_{g,2}$ sits as a factor of a stratum of $\mathcal{M}_{g,n}$, we have that $\textrm{diam}(\mathcal{M}_{g,2})\le \textrm{diam}(\mathcal{M}_{g,n})$. If $X$ is any point in $\mathcal{M}_{g+1,0}$, $X$ has a curve of length at most $4\log(8\pi g)$, (see e.g. Theorem 5.1.2 in \cite{Bus}) so using Wolpert's pinching formula
\[
\textrm{diam}(\mathcal{M}_{g+1,0})\le 2\sqrt{2\pi\log(8\pi g)}+\textrm{diam}(\mathcal{M}_{g,2})
\]
thus
\[
\textrm{diam}(\mathcal{M}_{g,n})\ge\textrm{diam}(\mathcal{M}_{g,2})\ge \textrm{diam}(\mathcal{M}_{g+1,0})-2\sqrt{2\pi\log(8\pi g)}
\]

Dividing by $h(g)$ and taking a limit, we see that $c$ is a lower bound for any limit point if $n$ is at least two. If $n=1$, then by a similar convexity argument to that given above, $\textrm{diam}(\mathcal{M}_{g-1,3})\le \textrm{diam}(\mathcal{M}_{g,1})$, so the result holds for $n=1$ as well.
\end{proof}


To finish proving Theorem \ref{thm:ngrowth}, it remains to prove that $\lim_{n\to\infty}\textrm{diam}(\mathcal{M}_{g,n})/\sqrt{n}$ exists and is independent of $g$. The existence of the limit follows from another convexity argument, while the independence requires the control on Bers' constant provided by Theorem \ref{thm:BPS} in the preliminaries.

\begin{proof}[Proof of Theorem \ref{thm:ngrowth}]

Let $F(n)=\textrm{diam}(\mathcal{M}_{0,n})$ as before. We have that the ratio $F(n)/\sqrt{n}$ is bounded above and below by positive constants, so let $C=\limsup F(n)/\sqrt{n}$. 

We can embed a product of $k$ copies of $M_{0,n}$ as strata in $\overline{M_{0,kn-2k+2}}$, which shows that $F(kn)\ge F(kn-2k+2)\ge\sqrt{k}F(n)$. Given $\varepsilon>0$, we show that $F(n)/\sqrt{n}> C-\varepsilon$ for sufficiently large $n$. Let $k$ be greater than $C\varepsilon/2$. Then for sufficiently large $n$,
\[
\frac{F(nk)}{\sqrt{nk}}-\frac{F(nk)}{\sqrt{(k+1)n}}\le F(nk)\frac{\sqrt{(k+1)n}-\sqrt{nk}}{\sqrt{nk}\sqrt{n(k+1)}}\le C\sqrt{nk}\frac{\sqrt{(k+1)n}-\sqrt{nk}}{\sqrt{nk}\sqrt{n(k+1)}}
\]
\[
=\frac{C n}{\sqrt{n(k+1)}}\frac{\sqrt{nk+n}-\sqrt{nk}}{n}=\frac{Cn}{\sqrt{n(k+1)}}\frac{x^{-1/2}}{2}
\]
for some $x\in[nk,nk+n]$ by the mean value theorem. This is bounded above by
\[
\frac{Cn}{\sqrt{n(k+1)}}\frac{(nk)^{-1/2}}{2}\le\frac{C}{2k}\le \varepsilon
\]

Let $m$ be such that $C-F(m)/\sqrt{m}\le \varepsilon$, and let $k$ be as above. Then for $n>mk$, we have that $n\in[mK,m(K+1)]$, so 
\[
\frac{F(n)}{\sqrt{n}}\ge \frac{F(mK)}{\sqrt{m(K+1)}}\ge \frac{F(mK)}{\sqrt{mK}}-\varepsilon\ge \frac{ \sqrt{K}F(m)}{\sqrt{mK}}-\varepsilon\ge C-2\varepsilon
\]

Thus $\lim_{n\to\infty} F(n)/\sqrt{n}$ exists. We will denote this limit $D$.

Let $F_g(n)=\textrm{diam}(\mathcal{M}_{g,n})$. To show that $\lim_{n\to\infty}F_g(n)/\sqrt{n}$ is independent of $g$, we note that for any $X$ in $\textrm{Teich}(\Sigma_{g,n})$ there is a pants decomposition on $X$ with every curve having length at most $A(g)\sqrt{n}$ by Theorem \ref{thm:BPS}. By pinching a nonseparating multicurve with $g$ components, we arrive at a stratum corresponding to a punctured sphere with $n+2g$ punctures. The total length of such a multicurve is at most $A(g)g\sqrt{n}$, so by Wolpert's pinching estimate $X$ is at most distance $\sqrt{2\pi gA(g)}n^{1/4}$ from such a stratum. This shows that

\[
\textrm{diam}(\mathcal{M}_{g,n})\le \textrm{diam}(\mathcal{M}_{0,n})+2\sqrt{2\pi gA(g)}n^{1/4} 
\]
so $\lim_{n\to\infty}F_g(n)/\sqrt{n}\le \lim_{n\to\infty}(F_0(n)/\sqrt{n}+2\sqrt{2\pi gA(g)}n^{-1/4})=D$. Since $\overline{\mathcal{M}_{g,n}}$ has a convex stratum which is a cover of $\overline{\mathcal{M}}_{0,n+2g}$, $F_0(n+2g)\le F_g(n)$ so clearly $C\le \lim_{n\to\infty}F_g(n)/\sqrt{n}$. 
\end{proof}

\end{document}